\documentclass[11pt,reqno,twoside]{amsart}
\usepackage{amssymb,amsmath,amsthm,soul,color,paralist}
\usepackage{t1enc}
\usepackage[cp1250]{inputenc}
\usepackage{a4,indentfirst,latexsym}
\usepackage{graphics}
\usepackage{mathrsfs}
\usepackage{cite,enumitem,graphicx}
\usepackage[colorlinks=true,urlcolor=blue,
citecolor=red,linkcolor=blue,linktocpage,pdfpagelabels,
bookmarksnumbered,bookmarksopen]{hyperref}
\usepackage[english]{babel}
\usepackage[left=2.1cm,right=2.1cm,top=2.72cm,bottom=2.72cm]{geometry}
\usepackage[metapost]{mfpic}
\usepackage[hyperpageref]{backref}
\usepackage[colorinlistoftodos]{todonotes}
\usepackage[normalem]{ulem}

\makeatletter
\providecommand\@dotsep{5}
\def\listtodoname{List of Todos}
\def\listoftodos{\@starttoc{tdo}\listtodoname}
\makeatother

\numberwithin{equation}{section}
\newtheorem{Th}{Theorem}[section]
\newtheorem{Prop}[Th]{Proposition}
\newtheorem{Lem}[Th]{Lemma}
\newtheorem{Cor}[Th]{Corollary}
\newtheorem{Def}[Th]{Definition}
\newtheorem{Rem}[Th]{Remark}

\newcommand{\eps}{\varepsilon}

\def\supp{\mathrm{supp}}

\def\N{\mathcal{N}}
\def\R{\mathbb{R}}

\newcommand{\h}{H^{s}_\eps}

\DeclareMathOperator{\e}{\varepsilon}

\newcommand{\C}{\mathbb{C}}

\begin{document}
\title[Nonlinear fractional magnetic Schr\"odinger equation]{Nonlinear fractional magnetic Schr\"odinger equation: existence and multiplicity}
\author[V. Ambrosio]{Vincenzo Ambrosio}
\address{Vincenzo Ambrosio\hfill\break\indent 
Dipartimento di Scienze Pure e Applicate (DiSPeA),\hfill\break\indent
Universit\`a degli Studi di Urbino `Carlo Bo'\hfill\break\indent
Piazza della Repubblica, 13\hfill\break\indent
61029 Urbino (Pesaro e Urbino, Italy)}
\email{\href{mailto:vincenzo.ambrosio@uniurb.it}{vincenzo.ambrosio@uniurb.it}}

\author[P. d'Avenia]{Pietro d'Avenia}
\address{Pietro d'Avenia\hfill\break\indent 
Dipartimento di Meccanica, Matematica e Management,\hfill\break\indent
Politecnico di Bari \hfill\break\indent
Via Orabona, 4\hfill\break\indent
70125 Bari (Italy)}
\email{\href{mailto:pietro.davenia@poliba.it}{pietro.davenia@poliba.it}}

\thanks{The authors are partially supported by grants of the group GNAMPA of INdAM}


\date{\today}

\begin{abstract}
In this paper we focus our attention on the following nonlinear fractional Schr\"odinger equation with magnetic field
\begin{equation*}
\varepsilon^{2s}(-\Delta)_{A/\varepsilon}^{s}u+V(x)u=f(|u|^{2})u \quad \mbox{ in } \mathbb{R}^{N},
\end{equation*}
where $\varepsilon>0$ is a parameter, $s\in (0, 1)$, $N\geq 3$, $(-\Delta)^{s}_{A}$ is the fractional magnetic Laplacian, $V:\mathbb{R}^{N}\rightarrow \mathbb{R}$ and $A:\mathbb{R}^{N}\rightarrow \mathbb{R}^N$ are continuous potentials and $f:\mathbb{R}^{N}\rightarrow \mathbb{R}$ is a subcritical nonlinearity. By applying variational methods and Ljusternick-Schnirelmann theory, we prove existence and multiplicity of solutions for $\varepsilon$ small.
\end{abstract}

\maketitle

\section{introduction}

\noindent 
In this paper we consider the following fractional nonlinear Schr\"odinger equation
\begin{equation}\label{P}
\e^{2s}(-\Delta)_{A/\e}^{s}u+V(x)u=f(|u|^{2})u \quad \mbox{ in } \R^{N}
\end{equation}
where $\e>0$ is a parameter, $s\in (0, 1)$, $N\geq 3$, $V\in C(\R^{N}, \R)$ and $A\in C^{0,\alpha}(\R^{N},\R^{N})$, $\alpha\in(0,1]$,  are the electric and magnetic potentials respectively,  $u\in\R^N\to\C$, $f: \R\rightarrow \R$. The fractional magnetic Laplacian is defined by
\begin{equation}\label{operator}
(-\Delta)^{s}_{A}u(x)
:=c_{N,s} \lim_{r\rightarrow 0} \int_{B_{r}^{c}(x)} \frac{u(x)-e^{\imath (x-y)\cdot A(\frac{x+y}{2})} u(y)}{|x-y|^{N+2s}} dy,
\quad
c_{N,s}:=\frac{4^{s}\Gamma\left(\frac{N+2s}{2}\right)}{\pi^{N/2}|\Gamma(-s)|}.
\end{equation}
This nonlocal operator has been defined in \cite{DS} as a fractional extension (for an arbitrary $s\in(0,1)$) of the magnetic pseudorelativistic operator, or {\em Weyl pseudodifferential operator defined with mid-point prescription},
\begin{align*}
{\mathscr H}_A u (x)
&=
\frac{1}{(2\pi)^3}\int_{\R^6} e^{\imath (x-y)\cdot \xi } \sqrt{\Big|\xi - A\big(\frac{x+y}{2}\big)\Big|^2} u(y) dyd\xi\\
&=
\frac{1}{(2\pi)^3}\int_{\R^6} e^{\imath (x-y)\cdot \left(\xi+ A\big(\frac{x+y}{2}\big)\right) } \sqrt{|\xi|^2} u(y) dyd\xi,
\end{align*}
introduced in \cite{IT} by Ichinose and Tamura, through oscillatory integrals, as a {\em fractional relativistic} generalization of the magnetic Laplacian (see also \cite{I10} and the references therein). Observe that for smooth functions $u$,
\begin{equation*}
\begin{split}
{\mathscr H}_A u (x)
&=
- \lim_{\eps\searrow 0}\int_{B^c_\eps(0)} \left[e^{-\imath y \cdot A \big( x+ \frac{y}{2}\big)} u(x+y) - u(x) - 1_{\{|y|<1\}}(y) y\cdot (\nabla - \imath A(x))u(x)\right]d\mu\\
&=
\lim_{\eps\searrow 0}\int_{B^c_\eps(x)} \left[u(x)-e^{\imath (x-y)\cdot A\left(\frac{x+y}{2}\right)}u(y)\right] \mu(y-x)dy ,
\end{split}
\end{equation*}
where
\[
d\mu=
\mu(y)dy = 
\frac{\Gamma\left(\frac{N+1}{2}\right)}{\pi^\frac{N+1}{2} |y|^{N+1}}dy.
\]
For details about the consistency of the definition in \eqref{operator} we refer the reader to \cite{NPSV,PSV1,PSV2,SV}.

The study of nonlinear fractional Schr\"odinger equations attracted a great attention, specially in the case $A=0$ (see \cite{MRS} and references therein). For instance, 
Felmer et al. \cite{FQT} dealt with existence, regularity and symmetry of positive solutions when $V$ is constant, and $f$ is a superlinear function with subcritical growth; see also \cite{A,DPPV, A3} and \cite{DSS} for the nonlocal Choquard equation.
Secchi \cite{Secchi} obtained the existence of ground state solutions under the assumptions that the potential $V$ is coercive. Shang and Zhang \cite{SZ} considered a fractional Schr\"odinger equation involving a critical nonlinearity, investigating the relation between the number of solutions and the topology of the set where $V$ attains its minimum.
Alves and Miyagaki \cite{AM} studied the existence and the concentration of positive solutions via penalization method (see also \cite{A1, A6, FigS, HZ} for related results).

On the other hand, the classical magnetic nonlinear Schr\"odinger equation has been extensively investigated by many authors \cite{AFF, AS, Cingolani, Cingolani-Secchi, EL, K} by applying suitable variational and topological methods.

However, in our nonlocal setting, only few papers \cite{DS,FPV,MPSZ,ZSZ} dealt with the existence and multiplicity of fractional magnetic problems. Therefore, motivated by this, in the present work we are interested in the existence and multiplicity of solutions to \eqref{P} when  the potential $V$ verifies the following condition 
\begin{equation}
\label{condV}
\tag{V}
V_{\infty}=\liminf_{|x|\rightarrow \infty} V(x)>V_{0}=\inf_{x\in \R^{N}} V(x)>0
\end{equation}
introduced by Rabinowitz in \cite{Rab}.

In this context, the presence of the nonlocal operator \eqref{operator} makes our analysis more complicated and intriguing, and new techniques are needed to overcome the difficulties that appear.

Before to state our results, we introduce the assumptions on the nonlinearity. Along the paper we will assume that $f: \R\rightarrow \R$ is a $C^{1}$ function satisfying the following assumptions:
\begin{enumerate}[label=($f_\arabic*$),ref=$f_\arabic*$]
	\item \label{f1}$f(t)=0$ for $t\leq 0$;
	\item \label{f2}$\displaystyle{\lim_{t\rightarrow 0} f(t)=0}$;
	\item \label{f3}there exists $q\in (2, 2^{*}_{s})$, where $2^{*}_{s}=2N/(N-2s)$, such that $\lim_{t\rightarrow \infty} f(t)/t^{\frac{q-2}{2}}=0$;
	\item \label{f4}there exists $\theta>2$ such that $0<\frac{\theta}{2} F(t)\leq t f(t)$ for any $t>0$, where $F(t)=\int_{0}^{t} f(\tau)d\tau$;
	\item \label{f5}there exists $\sigma\in (2, 2^{*}_{s})$ such that $f'(t)\geq C_{\sigma} t^{\frac{\sigma-4}{2}}$ for any $t>0$.
\end{enumerate} 

A first result we get is the following.
\begin{Th}\label{thex}
	Assume that \eqref{condV} and (\ref{f1})--(\ref{f5}) hold. Then there exists $\e_{0}>0$ such that the problem \eqref{Pe} admits a ground state solution for any $\e\in (0, \e_{0})$.
\end{Th}

Now, let us introduce the sets
\begin{equation}\label{defM}
M=\{x\in \R^{N}: V(x)=V_{0}\}
\quad
\hbox{and}
\quad
M_{\delta}=\{x\in \R^{N}: \operatorname{dist}(x, M)< \delta\}
\hbox{ for } \delta>0.
\end{equation}

In order to obtain a multiplicity result for \eqref{P}, we consider the Ljusternik-Schnirelmann category: given a closed set $Y$ is  of a topological space $X$, the Ljusternik-Schnirelmann category of $Y$ in $X$, denoted by $\operatorname{cat}_{X}(Y)$, is the least number of closed and contractible sets in $X$ which cover $Y$ (see \cite{W}).\\
More precisely we have
\begin{Th}\label{thmf}
	Assume $V$ verifies \eqref{condV}, and $f$ satisfies (\ref{f1})--(\ref{f5}). Then, for any $\delta>0$ there exists $\e_{\delta}>0$ such that, for any $\e\in (0, \e_{\delta})$, the problem  \eqref{P} has at least $\operatorname{cat}_{M_{\delta}}(M)$ nontrivial solutions. 
\end{Th}
The proof of the above theorem is based on variational methods.

In the study of our problem, we will use the diamagnetic inequality recently established in \cite{DS} and some interesting decay properties of positive solutions to the limit problem associated to \eqref{P} (see \cite{FQT}).
These facts combined with the H\"older continuity assumption on the magnetic potential, will play an essential role to get some useful estimates needed to obtain the existence of  solutions and to implement the barycenter machinery.

The paper is organized as follows: in Section \ref{sec2} we introduce the functional setting and we give some fundamental tools and in Sections \ref{sec3} and \ref{sec4} we give the proof of Theorems \ref{thex} and \ref{thmf} respectively.

\subsection*{Notations:}
In what follows $|\cdot|_r$ denotes the $L^r(\R^N)$ norm, $\Re(z)$ is the real part of the complex number $z$, the letters $C$, $C_i$ will be repeatedly used to denote various positive constants whose exact values are irrelevant and can change from line to line, and $B_R(x)$ is the ball in $\R^N$ centered at $x$ with radius $R$.

\section{The space $H^{s}_{\e}$}\label{sec2}

By using the change of variable $x\mapsto \e x$ we can see that the problem (\ref{P}) is equivalent to the following one
\begin{equation}\label{Pe}
(-\Delta)_{A_{\e}}^{s} u + V_\eps( x)u =  f(|u|^{2})u  \mbox{ in } \R^{N},
\end{equation}
where $A_{\e}(x)=A(\e x)$ and $V_\eps(x)=V(\eps x)$.


For a function $u:\R^N\to\C$, let us denote by
$$
[u]^{2}_{A}:=\frac{c_{N,s}}{2}\iint_{\R^{2N}} \frac{|u(x)-e^{\imath (x-y)\cdot A(\frac{x+y}{2})} u(y)|^{2}}{|x-y|^{N+2s}} \, dxdy,
$$
and consider
$$
D_A^s(\R^N,\C)
:=
\left\{
u\in L^{2_s^*}(\R^N,\C) : [u]^{2}_{A}<\infty
\right\}.
$$
Then let us introduce the Hilbert space
$$
H^{s}_{\e}:=
\left\{
u\in D_{A_\eps}^s(\R^N,\C): \int_{\R^{N}} V(\e x) |u|^{2}\, dx <\infty
\right\}
$$ 
endowed with the scalar product
\begin{align*}
\langle u , v \rangle_\eps
&:=
\Re	\int_{\R^{N}} V(\e x) u \bar{v} dx\\
&\qquad
+ \frac{c_{N,s}}{2}\Re\iint_{\R^{2N}} \frac{(u(x)-e^{\imath(x-y)\cdot A_{\e}(\frac{x+y}{2})} u(y))\overline{(v(x)-e^{\imath(x-y)\cdot A_{\e}(\frac{x+y}{2})}v(y))}}{|x-y|^{N+2s}} dx dy
\end{align*}
and let
$$
\|u\|_{\e}:=\sqrt{\langle u , u \rangle_\eps}.
$$

Observe that for $A=0$ we recover the classical definition of $H^s(\R^N,\C)$ (for details we refer the reader to \cite{DPV}).

If $u\in \h$, let
\begin{equation}\label{eqtruncation}
\hat{u}_{j}(x):=\varphi_{j}(x)u(x)
\end{equation}
where $j\in\mathbb{N}^*$ and $\varphi_{j}(x)=\varphi(2x/j)$ with $\varphi\in C^{\infty}_{0}(\R^{N}, \R)$, $0\leq \varphi\leq 1$, $\varphi(x)=1$ if $|x|\leq 1$, and $\varphi(x)=0$ if $|x|\geq 2$. Note that $\hat{u}_{j}\in \h$ and $\hat{u}_{j}$ has compact support.\\
Proceeding as in \cite[Lemma 3.2]{ZSZ}, we get the following useful result.
\begin{Lem}\label{truncation}
	For any $\e>0$, it holds $\|\hat{u}_{j}-u\|_{\e}\rightarrow 0$ as $j\rightarrow \infty$.
\end{Lem}

The space $\h$ satisfies the following fundamental properties.
\begin{Lem}
The space $\h$ is complete and $C_c^\infty(\R^N,\C)$ is dense in $\h$. 
\end{Lem}
\begin{proof}
To prove that $\h$ is a complete space, let us consider a Cauchy sequence $(u_{n})$  in $H^{s}_{A_{\e}}$.
In particular $(\sqrt{V_\eps} u_{n})$ is a Cauchy sequence in $L^{2}(\R^{N}, \C)$, and being $V_\eps\geq V_{0}$ in $\R^N$, there exists $u\in L^{2}(\R^{N}, \C)$ such that $\sqrt{V_\eps} u_{n} \rightarrow \sqrt{V_\eps}u$ in $L^{2}(\R^{N}, \C)$ and a.e. in $\R^{N}$. By using Fatou's Lemma
we get $u_{n}\rightarrow u$ in $\h$.\\
To prove that $C^{\infty}_{c}(\R^{N}, \C)$ is dense in $\h$ we fix $u\in \h$ and we consider the sequence $\hat{u}_{j}(x)=u(x)\varphi({x}/{j})$  defined as in \eqref{eqtruncation}.\\
In view of Lemma \ref{truncation}, we know that  $\|\hat{u}_{j} -u \|_{\e}\rightarrow 0$ as $j\rightarrow \infty$ and so it is enough to prove the density for compact supported functions in $\h$. \\
Now, we consider $v\in \h$ with compact support, and assume that $\operatorname{supp}(v)\subset B_{R}(0)$. Taking into account
$$
|u(x)-u(y)|^{2}\leq 2 |u(x)-u(y)e^{\imath A_{\e}(\frac{x+y}{2})\cdot (x-y)}|^{2}+2|u(y)|^{2}|e^{\imath A_{\e}(\frac{x+y}{2})\cdot (x-y)}-1|^{2}
$$
and that, from $|e^{\imath t}-1|^{2}\leq 4$ and $|e^{\imath t}-1|^{2}\leq t^2$, we deduce
\begin{align*}
\int_{B_{R}(0)} |u(y)|^{2}dy \int_{\R^{N}} \frac{|e^{\imath A_{\e}(\frac{x+y}{2})\cdot (x-y)}-1|^{2}}{|x-y|^{N+2s}} dx 
&\leq
C\left[\int_{B_{R}(0)} |u(y)|^{2}dy \int_{|x-y|>1} \frac{1}{|x-y|^{N+2s}} dx\right.\\
&\
\left.+ \int_{B_{R}(0)} |u(y)|^{2}dy\int_{|x-y|\leq 1} \frac{\max_{|z|\leq \frac{2R+1}{2}}|A_{\e}(z)|^{2}}{|x-y|^{N+2s-2}} dx\right] \\
&<\infty,
\end{align*}
since $V_\eps \geq V_{0}$ in $\R^{N}$, we can see that $u\in H^{s}(\R^{N}, \C)$.\\
Then, it makes sense to define $u_{\e}=\rho_{\e}*u\in C^{\infty}_{c}(\R^{N}, \C)$, where $\rho_{\e}$ is a mollifier with $\supp(\rho_{\e})\subset B_\eps (0)$. Arguing as in \cite[Theorem 3.24]{HT}
we have that $u_{\e}\rightarrow u$ in $H^{s}(\R^{N}, \C)$ as $\e\rightarrow 0$.\\
Moreover there exists $K>0$ such that $\operatorname{supp}(u_{\e}-u)\subset B_{K}(0)$ for all $\e>0$ small enough and, arguing as before,
\begin{align*}
[u_{\e}-u]^{2}_{A_{\e}}
&\leq 2[u_{\e}-u]^{2}+2\iint_{\R^{2N}} |(u_{\e}-u)(y)|^{2}  \frac{|e^{\imath A_{\e}(\frac{x+y}{2})\cdot (x-y)}-1|^{2}}{|x-y|^{N+2s}} dxdy \\
&\leq 2[u_{\e}-u]^{2}
+C\left[\int_{B_{K}(0)} |(u_{\e}-u)(y)|^{2} dy \int_{|x-y|>1} \frac{1}{|x-y|^{N+2s}} dx\right.\\
&\qquad
+\left. \int_{B_{K}(0)} |(u_{\e}-u)(y)|^{2} dy\int_{|x-y|\leq 1} \frac{(\max_{|z|\leq \frac{2K+1}{2}}|A_{\e}(z)|)^{2}}{|x-y|^{N+2s-2}} dx \right] \\
&\leq 2[u_{\e}-u]^{2}+C\int_{B_{K}(0)} |(u_{\e}-u)(y)|^{2} dy \rightarrow 0 \mbox{ as } \e\rightarrow 0.
\end{align*}
\end{proof}

Using \eqref{condV} and the pointwise diamagnetic inequality
\[
||u(x)| - |u(y)||\leq \left|u(x)-e^{\imath (x-y)\cdot A\left(\frac{x+y}{2}\right)}u(y)\right|,
\]
we can proceed as in \cite[Lemma 3.1]{DS} to prove that if $u\in H_\eps^s$, then $|u|\in H^s(\R^N,\R)$ and the following fractional diamagnetic inequality
\begin{equation}
\label{eqDI}
[|u|]^2 \leq [u]^{2}_{A_{\e}}
\end{equation}
holds, where
\[
[u]^2 := \frac{c_{N,s}}{2}\iint_{\R^{2N}} \frac{|u(x)- u(y)|^{2}}{|x-y|^{N+2s}} \, dxdy.
\]
Then, arguing as in \cite[Lemma 3.5]{DS} and using \cite[Lemma 3.2]{Cheng}, we get 
\begin{Lem}\label{embedding}
	The space $H^{s}_{\e}$ is continuously embedded in $L^{r}(\R^{N}, \C)$ for $r\in [2, 2^{*}_{s}]$, and compactly embedded in $L_{\rm loc}^{r}(\R^{N}, \C)$ for $r\in [1, 2^{*}_{s})$.\\
	Moreover, if $V_\infty=\infty$, then, for any bounded sequence $(u_{n})$  in $\h$, we have that, up to a subsequence, $(|u_{n}|)$ is strongly convergent in $L^{r}(\R^{N}, \R)$ for $r\in [2, 2^{*}_{s})$.
\end{Lem}


For compact supported functions in $H^{s}(\R^{N}, \R)$ we can prove the following result.
\begin{Lem}\label{aux}
If $u\in H^{s}(\R^{N}, \R)$ and $u$ has compact support, then $w=e^{\imath A(0)\cdot x} u \in \h$.
\end{Lem}
\begin{proof}
	Assume that $\operatorname{supp}(u)\subset B_{R}(0)$.
	Since $V$ is continuous it is clear that
	$$
	\int_{\R^{N}} V(\e x)|w|^{2} dx=\int_{B_{R}(0)} V(\e x)|u|^{2} dx\leq C|u|_2^2<\infty.
	$$ 
	Therefore, it is enough to show that $[w]_{A_\eps}<\infty$.\\
	Recalling that $A$ is continuous and $|e^{\imath t}-1|^{2}\leq 4$ and $|e^{\imath t}-1|^{2}
	\leq t^{2}$ for all $t\in \R$, we have
	\begin{align*}
	[w]_{\e}^{2}
	&=\iint_{\R^{2N}} \frac{|e^{\imath A(0)\cdot x}u(x)-e^{\imath A(0)\cdot y} e^{\imath A_{\e}(\frac{x+y}{2})\cdot (x-y)}u(y)|^{2}}{|x-y|^{N+2s}} dx dy  \\
	&\leq 2 [u]^{2}+2\iint_{\R^{2N}} \frac{u^{2}(y) |e^{\imath [A_{\e}(\frac{x+y}{2})-A(0)]\cdot (x-y)}-1|^{2}}{|x-y|^{N+2s}} dx dy \\
	&\leq 2 [u]^{2}+2 \int_{B_{R}(0)} u^{2}(y) dy \Bigl[\int_{|x-y|\geq 1} \frac{4}{|x-y|^{N+2s}} dx +\int_{|x-y|<1}  \frac{|A_{\e}(\frac{x+y}{2})-A(0)|^{2}}{|x-y|^{N+2s-2}} dx\Bigr] \\
	&\leq 2 [u]^{2}+2 \int_{B_{R}(0)} u^{2}(y) dy \Bigl[\int_{|x-y|\geq1} \frac{4}{|x-y|^{N+2s}} dx +\int_{|x-y|<1}  \frac{(\max_{|z|\leq \frac{2R+1}{2}}[|A_{\e}(z)|+|A(0)|])^{2}}{|x-y|^{N+2s-2}} dx\Bigr] \\
	&\leq 2 [u]^{2}+C \int_{B_{R}(0)} u^{2}(y) dy \Bigl[\int_{1}^{\infty} \frac{1}{\rho^{2s+1}} d\rho + \int_{0}^{1}  \frac{1}{\rho^{2s-1}} d\rho\Bigr]<\infty
	\end{align*}
	because of $u\in H^{s}(\R^{N}, \R)$ and $s\in (0, 1)$.
\end{proof}

Moreover we have the following Lions-type Lemma (see \cite[Lemma 2.2]{FQT}).
\begin{Lem}
	\label{Lions}
	Let $N\geq 2$. If $(u_{n})$ is a bounded sequence in $H^{s}(\R^{N}, \R)$ and if 
	\begin{equation*}
	\lim_{n} \sup_{y\in \R^{N}} \int_{B_{R}(y)} |u_{n}|^{2} dx=0
	\end{equation*}
	where $R>0$, then $u_{n}\rightarrow 0$ in $L^{r}(\R^{N}, \R)$ for all $r\in (2, 2^{*}_{s})$.
\end{Lem}

Arguing as in \cite[Lemma 3.2]{DL} and taking into account Lemma \ref{embedding} we can prove
\begin{Lem}\label{vanishing}
	Let $\tau\in [2, 2^{*}_{s})$ and $(u_{n})\subset \h$ be a bounded sequence. Then there exists a subsequence $(u_{n_{j}})\subset \h$ such that for any $\sigma>0$ there exists $r_{\sigma,\tau}>0$ such that 
	\begin{equation}\label{DL}
	\limsup_{j} \int_{B_{j}(0)\setminus B_{r}(0)} |u_{n_{j}}|^{\tau} dx\leq \sigma
	\end{equation}
	for any $r\geq r_{\sigma}$.
\end{Lem}

We conclude this section giving some properties on the nonlinearity that will be useful in the proofs of our results.

\begin{Lem}
\label{propf}
The nonlinearity satisfies the following properties:
\begin{enumerate}[label=(\roman*),ref=\roman*]
	\item \label{propf1} for every $\xi >0$ there exists $C_{\xi}>0$ such that for all $t\in\R$,
	$$\frac{\theta}{2}F(t^{2}) \leq f(t^2)t^2\leq \xi t^{2} + C_{\xi} |t|^{q};$$
	\item \label{propf2} there exist $C_1,C_2>0$ such that for all $t\in\R$, $F(t^{2})\geq C_{1}|t|^{\vartheta} - C_{2}$;
	\item \label{propf3} if $u_{n_j}\rightharpoonup u$ in $\h$ and $\hat{u}_j$ is defined as in \eqref{eqtruncation} we have that
	$$
	\int_{\R^{N}} F(|u_{n_j}|^{2})-F(|u_{n_j}-\hat{u}_{j}|^{2})-F(|\hat{u}_{j}|^{2}) dx=o_{j}(1)
	\quad
	\hbox{as }j\to\infty;
	$$
	\item \label{propf4} if $(u_n)\subset\h$ is bounded, $(u_{n_j})$ a subsequence as in Lemma \ref{vanishing} such that $u_{n_j}\rightharpoonup u$ in $\h$ and $\hat{u}_j$ is defined as in \eqref{eqtruncation} we have that
	\[
\int_{\R^{N}} [f(|u_{n_{j}}|^{2})u_{n_{j}}-f(|u_{n_j}-\hat{u}_{j}|^{2})(u_{n_j}-\hat{u}_{j})-f(|\hat{u}_{j}|^{2})\hat{u}_{j}] \phi dx \to 0
\quad
\hbox{as }j\to\infty
\]
uniformly with respect to $\phi\in \h$ with $\|\phi\|_{\e}\leq 1$.
\end{enumerate}
\end{Lem}
\begin{proof}
Properties (\ref{propf1}) and (\ref{propf2}) are easy consequences of (\ref{f2}), (\ref{f3}) and (\ref{f4}). \\
Let us prove (\ref{propf3}). Recalling that $\hat{u}_{j}=\varphi_{j}u$ with $\varphi_{j}\in [0, 1]$, (\ref{propf1}) in Lemma \ref{propf}, and using the Young inequality we can see that
\begin{align*}
|F(|u_{n_j}|^{2})-F(|u_{n_j}-\hat{u}_{j}|^{2})|
&\leq
2 \int_{0}^{1}  |f(|u_{n_j}-t\hat{u}_{j}|^{2})| |u_{j}-t\hat{u}_{j}| |\hat{u}_{j}| dt\\
&\leq
C \left[(|u_{n_j}|+|u|) |u|+(|u_{n_j}|+|u|)^{q-1}|u|\right]\\
&\leq
\xi (|u_{n_j}|^{2}+|u_{n_j}|^{q})+C(|u|^{2}+|u|^{q})
\end{align*}
for any $\xi>0$. Then
\[
|F(|u_{n_j}|^{2})-F(|u_{n_j}-\hat{u}_{j}|^{2})-F(|\hat{u}_{j}|^{2})|
\leq
\xi (|u_{n_j}|^{2}+|u_{n_j}|^{q})+C(|u|^{2}+|u|^{q})
\]
Now let
$$
G_{j}^\xi:=\max\left\{|F(|u_{n_j}|^{2})-F(|u_{n_j}-\hat{u}_{j}|^{2})-F(|\hat{u}_{j}|^{2})|-\xi(|u_{n_j}|^{2}+|u_{n_j}|^{q}), 0\right\}.
$$
Note that $G_{j}^\xi\rightarrow 0$ as $j\rightarrow \infty$ a.e. in $\R^N$ and $0\leq G_{j}^\xi\leq C(|u|^{2}+|u|^{q})\in L^{1}(\R^{N}, \R)$. Thus, applying the Dominated Convergence Theorem, we deduce that
$$
\int_{\R^{N}} G_{j}^\xi dx\rightarrow 0
\quad
\hbox{as }j\to\infty.
$$
On the other hand, from the definition of $G_{j}^\xi$, 
$$
|F(|u_{n_j}|^{2})-F(|u_{n_j}-\hat{u}_{j}|^{2})-F(|\hat{u}_{j}|^{2})|\leq \xi(|u_{j}|^{2}+|u_{j}|^{2^{*}_{s}})+G_{j}^\xi.
$$
Hence, since $(u_{n_j})$ is bounded in $\h$, we have
$$
\limsup_{j}\int_{\R^{N}} |F(|u_{n_j}|^{2})-F(|u_{n_j}-\hat{u}_{j}|^{2})-F(|\hat{u}_{j}|^{2})|\leq C\xi
$$
and, from the arbitrariness of $\xi$, we conclude.\\
To prove (\ref{propf4}), let us consider $\phi\in \h$ such that $\|\phi\|_{\e}\leq 1$ and $\sigma>0$. Note that, for any $r\geq\max\{r_{\sigma,2},r_{\sigma,q}\}$, where $r_{\sigma,\tau}$ has been introduced in Lemma \ref{DL},
\begin{equation*}
\begin{split}
&\left|\int_{\R^{N}} [f(|u_{n_{j}}|^{2})u_{n_{j}}-f(|u_{n_j}-\hat{u}_{j}|^{2})(u_{n_j}-\hat{u}_{j})-f(|\hat{u}_{j}|^{2})\hat{u}_{j}] \phi dx \right|\\
&\qquad \leq \int_{B_{r}(0)} |f(|u_{n_{j}}|^{2})u_{n_{j}}-f(|v_{j}|^{2})v_{j}-f(|\hat{u}_{j}|^{2})\hat{u}_{j}| |\phi| dx\\
&\qquad + \int_{B_{r}^c(0)} |f(|u_{n_{j}}|^{2})u_{n_{j}}-f(|v_{j}|^{2})v_{j}-f(|\hat{u}_{j}|^{2})\hat{u}_{j}| |\phi| dx\\
&\qquad =:D_{j}+E_{j}.
\end{split}
\end{equation*}
Taking into account Lemma \ref{embedding} and Lemma \ref{truncation}, we can apply the Dominated Convergence Theorem to obtain that $D_{j}\rightarrow 0 \mbox{ uniformly in } \phi\in \h \hbox{ with } \|\phi\|_{\e}\leq 1$.\\
On the other hand, recalling that (\ref{propf1}) in Lemma \ref{propf} and that $\hat{u}_{j}=0$ in $B_{j}^c(0)$ for any $j\geq 1$, we deduce that, for $j$ large enough,
\begin{equation*}
\begin{split}
E_{j}
&=\int_{B_{j}(0)\setminus B_{r}(0)} |f(|u_{n_{j}}|^{2})u_{n_{j}}-f(|u_{n_{j}}-\hat{u}_{j}|^{2})(u_{n_{j}}-\hat{u}_{j})-f(|\hat{u}_{j}|^{2})\hat{u}_{j}| |\phi| dx \\
&\leq C \int_{B_{j}(0)\setminus B_{r}(0)}(|u_{n_{j}}|+|\hat{u}_{j}|+|u_{n_{j}}|^{q-1}+|\hat{u}_{j}|^{q-1})|\phi| dx.
\end{split}
\end{equation*}
Since $\|\phi\|_{\e}\leq 1$, using also the H\"older inequality and Lemma \ref{embedding}, we get
\begin{equation*}
\int_{B_{j}(0)\setminus B_{r}(0)}(|u_{n_{j}}|+|u_{n_{j}}|^{q-1}) |\phi| dx 
\leq
C \left[
\left(\int_{B_{j}(0)\setminus B_{r}(0)} |u_{n_{j}}|^{2}dx \right)^{\frac{1}{2}}
+ \left(\int_{B_{j}(0)\setminus B_{r}(0)} |u_{n_{j}}|^{q}dx \right)^{\frac{q-1}{q}} 
\right]
\end{equation*}
and so, by Lemma \ref{DL},
\begin{equation*}
\limsup_{j}  \int_{B_{j}(0)\setminus B_{r}(0)}(|u_{n_{j}}|+|u_{n_{j}}|^{q-1}) |\phi| dx 
\leq C(\sigma^{\frac{1}{2}}+\sigma^{\frac{q-1}{q}}).
\end{equation*}
Moreover, note that from Lemma \ref{embedding} and Lemma \ref{truncation}, we know that $\hat{u}_{j}\rightarrow u$ in $L^{2}(\R^{N}, \C)\cap L^{q}(\R^{N}, \C)$ as $j\rightarrow \infty$.
This and H\"older inequality give
\begin{equation*}
\limsup_{j}  \int_{B_{j}(0)\setminus B_{r}(0)}(|\hat{u}_{j}|+|\hat{u}_{j}|^{q-1}) |\phi| dx= \int_{B_{r}^c(0)} (|u|+|u|^{q-1})|\phi| dx \leq C(\sigma^{\frac{1}{2}}+\sigma^{\frac{q-1}{q}})
\end{equation*}
for $r$ large enough.
Thus the arbitrariness of $\sigma>0$ yields $E_{j}\rightarrow 0$ as $ j\rightarrow \infty$ uniformly with respect to $\phi$, $\|\phi\|_{\e}\leq 1$ and we conclude.
\end{proof}

\section{A first existence result}\label{sec3}

The goal of this section is to prove Theorem \ref{thex}.

We want to find solutions of \eqref{Pe} in the sense of the following definition.
\begin{Def}
	We say that $u\in H^{s}_{\e}$ is a weak solution to \eqref{Pe} if for any $v\in H^{s}_{\e}$
	\begin{align*}
	& \Re \Bigl(\frac{c_{N,s}}{2} \iint_{\R^{2N}} \frac{(u(x)-e^{\imath (x-y)\cdot A_{\e}(\frac{x+y}{2})} u(y)) \overline{(v(x)-e^{\imath (x-y)\cdot A_{\e}(\frac{x+y}{2})} v(y))}}{|x-y|^{N+2s}} \, dxdy \\
	&\qquad +\int_{\R^{N}} V(\e x) u \bar{v}\, dx-\int_{\R^{N}} f(|u|^{2})u\bar{v}\, dx\Bigr)=0.
	\end{align*}
\end{Def}

Such solutions can be found as critical points of the  functional $J_{\e}: \h\rightarrow \R$ defined as
\[
J_{\e}(u)
=\frac{1}{2} \|u\|^{2}_{\e}-\frac{1}{2} \int_{\R^{N}} F(|u|^{2})\, dx.
\]
Using Lemma \ref{embedding} and Lemma \ref{propf}, we can get  that $J_{\e}$ is well-defined and that $J_{\e}\in C^{1}(\h, \R)$.

Let us show that for any $\e>0$ the functional $J_\eps$ satisfies the geometrical assumptions of the Mountain Pass Theorem.

\begin{Lem}\label{MPG}
	The functional $J_{\e}$ satisfies the following conditions:
	\begin{enumerate}[label=(\roman*),ref=\roman*]
		\item \label{iMPG} there exist $\alpha, \rho >0$ such that $J_{\e}(u)\geq \alpha$ with $\|u\|_{\e}=\rho$; 
		\item \label{iiMPG}there exists $e\in \h\setminus B_{\rho}(0)$ such that $J_{\e}(e)<0$. 
	\end{enumerate}
\end{Lem}

\begin{proof}
	Taking into account (\ref{propf1}) in Lemma \ref{propf}, Lemma \ref{embedding}, and \eqref{condV}, for $\xi<V_{0}$ we get
	\[
	J_{\e}(u) 
	\geq \frac{1}{2} [u]_{A_\eps}^2
	+\frac{1}{2} \left(1-\frac{\xi}{V_0}\right)\int_{\R^{N}} V(\e x) |u|^{2} dx
	-\frac{C_{\xi}}{2}\int_{\R^{N}} |u|^{q} dx
	\geq C_1 \|u\|_{\e}^{2} - C_2 \|u\|_{\e}^{q}
	\]
	and then (\ref{iMPG}).\\
	To prove (\ref{iiMPG}), we observe that by (\ref{propf2}) in Lemma \ref{propf}
	and taking $\varphi \in C^{\infty}_{c}(\R^{N}, \C)$ such that $\varphi \not \equiv 0$ we have 
	\begin{equation*}
	J_{\e}(t\varphi)
	\leq
	\frac{t^{2}}{2} \|\varphi\|_{\e}^{2}
	- t^{\vartheta}C_{1}  |\varphi|^{\vartheta}_\vartheta
	+ C_{2}|\supp (\varphi)|\rightarrow -\infty \mbox{ as } t\rightarrow +\infty
	\end{equation*}
	since $\vartheta >2$.
\end{proof}

By the Ekeland Variational Principle there exists a $(PS)_{c_{\e}}$ sequence $(u_{n})\subset \h$, that is 
\begin{equation}\label{PSc}
J_{\e}(u_{n})\rightarrow c_{\e} \quad \mbox{ and }\quad J'_{\e}(u_{n})\rightarrow 0, 
\end{equation}
where $c_{\e}$ is the minimax level of the Mountain Pass Theorem, namely
\[
c_\eps:= \inf_{\gamma\in\Gamma} \max_{t\in [0,1]} J_\eps(\gamma(t))
\]
with $\Gamma:=\{\gamma\in H([0,1],H_\eps^s) : \gamma(0)=0,J_\eps(\gamma(1))<0\}$.
\\
Let us observe that $(u_{n})$ is bounded in $\h$. In fact by using \eqref{PSc} and (\ref{f4}) we can see that
\begin{align*}
c_{\e}+o_{n}(1)\|u_{n}\|_{\e}&=J_{\e}(u_{n})-\frac{1}{\theta}\langle J'_{\e}(u_{n}),u_{n} \rangle \\
&= \left(\frac{1}{2}-\frac{1}{\theta}\right) \|u_{n}\|^{2}_{\e}+\int_{\R^{N}} \left[\frac{1}{\theta} f(|u_{n}|^{2})|u_{n}|^{2}-\frac{1}{2} F(|u_{n}|^{2})\right] dx \\
&\geq \left(\frac{1}{2}-\frac{1}{\theta}\right) \|u_{n}\|^{2}_{\e}.
\end{align*}
Moreover it is standard
to verify the characterization
\begin{equation*}
c_{\e}= \inf_{u\in \h\setminus \{0\}} \sup_{t\geq 0} J_{\e}(tu) = \inf_{u\in \N_{\e}} J_{\e}(u), 
\end{equation*}
where
$$
\N_{\e}:=\{u\in \h\setminus\{0\}: \langle J'_{\e}(u),u\rangle=0\}
$$
is the usually Nehari manifold associated to $J_\eps$. \\
The following properties hold.
\begin{Lem}\label{LemNeharyE}
We have:
\begin{enumerate}[label=(\roman*),ref=\roman*]
	\item \label{2.1} there exists $K>0$ such that,  for all $u\in \N_{\e}$, $\|u\|_{\e}\geq K$;
	\item \label{2.2} for any $u\in \h\setminus \{0\}$ there exists a unique $t_{0}= t_{0}(u)$ such that $J_{\e}(t_{0}u)= \max_{t\geq 0} J_{\e}(tu)$ and then $t_{0}u\in \N_{\e}$. 
\end{enumerate}
\end{Lem}
\begin{proof}
Property (\ref{2.1}) follows easily from (\ref{propf1}) in Lemma \ref{propf} and Lemma \ref{embedding}, since, if $u\in \N_{\e}$, then, for all $\xi >0$
\[
\|u\|_\eps^2 = \int_{\R^{N}} f(|u|^{2})|u|^{2} dx \leq \xi \|u\|_\eps^2 + C \| u\|_\eps^q.
\]
To prove (\ref{2.2}), let us fix $u\in \h\setminus \{0\}$ and consider the smooth function $h(t):=J_\eps(tu)$ for $t\geq 0$. Arguing as in Lemma \ref{MPG} we can get that
\[
J_{\e}(tu) 
\geq C_1 t^2 \|u\|_{\e}^{2} - C_2 t^q \|u\|_{\e}^{q}
\]
and
\begin{equation*}
J_{\e}(tu)
\leq
\frac{t^{2}}{2} \|u\|_{\e}^{2}
- t^{\vartheta}C_{1}  \int_{\Omega}|u|^\vartheta dx
+ C_{2}|\Omega|\rightarrow -\infty \mbox{ as } t\rightarrow +\infty,
\end{equation*}
where $\Omega$ is a compact subset of $\operatorname{supp}(u)$ with $|\Omega|>0$. Then
there exists a maximum point of $h$. To prove the uniqueness, let $0<t_1<t_2$ be two maximum points of $h$. Since $h'(t_1)=h'(t_2)=0$, then
\[
\| u\|_\eps^2 = \int_{\R^{N}} f(|t_1 u|^{2})|u|^{2} dx = \int_{\R^{N}} f(|t_2 u|^{2})|u|^{2} dx
\]
which is in contradiction with the strict increasing of $f$ assumed in (\ref{f5}).
\end{proof}

To prove the compactness of the $(PS)_d$ sequences, for suitable $d\in\R$, we will use the following preliminary result.

\begin{Lem}\label{compactness}
	Let $d\in\R$ and $(u_{n})\subset \h$ be a $(PS)_{d}$ sequence for $J_{\e}$ such that $u_{n}\rightharpoonup 0$ in $\h$. Then, one of the following alternatives occurs:
	\begin{enumerate}[label=(\alph*),ref=\alph*]
		\item \label{compactnessa} $u_{n}\rightarrow 0$ in $\h$; 
		\item \label{compactnessb} there are a sequence $\{y_{n}\}\subset \R^{N}$ and constants $R, \beta>0$ such that 
		\begin{equation*}
		\liminf_{n} \int_{B_{R}(y_{n})} |u_{n}|^{2} dx \geq \beta >0. 
		\end{equation*}
	\end{enumerate}
\end{Lem}

\begin{proof}
	Assume that (\ref{compactnessb}) does not hold true. Then, for every $R>0$ such that
	\begin{equation*}
	\lim_{n} \sup_{y\in \R^{N}} \int_{B_{R}(y)} |u_{n}|^{2} dx=0. 
	\end{equation*}
	Since $(u_{n})$ is bounded in $\h$, from \eqref{eqDI} we deduce that $(|u_{n}|)$ is bounded in $H^{s}(\R^{N}, \R)$, so by Lemma \ref{Lions} it follows that $	|u_{n}|_q\rightarrow 0$.\\
	Since, moreover, $(u_{n})$ is also a $(PS)_{d}$ sequence for $J_{\e}$, by (\ref{propf1}) in Lemma \ref{propf} we have that for every $\xi>0$
	\[
	0\leq \|u_{n}\|_{\e}^{2}
	= \int_{\R^{N}} f(|u_{n}|^{2})|u_{n}|^{2} dx + o_{n}(1) 
	\leq \xi |u_{n}|_2^{2} + C_{\xi} |u_{n}|_q^{q} + o_{n}(1)
	\leq \frac{\xi}{V_0} \|u_{n}\|_\eps^{2} + C_{\xi} |u_{n}|_q^{q} + o_{n}(1).
	\]
	Thus, for $\xi$ small enough, we get (\ref{compactnessa}).
\end{proof}

Moreover, to develop our arguments, we will need to consider the following family of limit problems associated to \eqref{Pe}
\begin{equation}
\label{Plim}
\tag{$P_{\mu}$}
(-\Delta)^{s} u + \mu u =  f(|u|^{2})u  \mbox{ in } \R^{N},
\end{equation}
with $\mu>0$, whose corresponding $C^1$ functional $I_{V_0}: H^{s}(\R^N,\R)\rightarrow \R$ is given by
\[
I_{\mu}(u)
=\frac{1}{2} \|u\|^{2}_{\mu}-\frac{1}{2} \int_{\R^{N}} F(|u|^{2})\, dx,
\]
where
\[
\|u\|^{2}_{\mu}:=[u]^2+V_{0} |u|_2^{2}.
\]
Even in this case we can define the Nehari manifold 
$$
\mathcal{M}_{\mu}=\{u\in H^{s}(\R^N,\R): \langle I'_{\mu}(u),u\rangle=0\}
$$
and we have that
$$
c_{\mu}
:= \inf_{\gamma\in\Xi_\mu} \max_{t\in [0,1]} I_{\mu}(\gamma(t))
=\inf_{u\in H^{s}(\R^N,\R)\setminus \{0\}} \sup_{t\geq 0} I_{\mu}(t u)
=\inf_{u\in \mathcal{M}_{\mu}} I_{\mu}(u)
$$
with $\Xi_\mu:=\{\gamma\in C([0,1],H^s(\R^N,\R)) : \gamma(0)=0,I_\mu(\gamma(1))<0\}$.\\
We will call {\em ground state} for \eqref{Plim} each minimum of $I_\mu$ in $\mathcal{M}_{\mu}$, wich is also a solution of \eqref{Plim}.

\begin{Rem}
Arguing as in Lemma \ref{LemNeharyE} we can prove that for every fixed $\mu>0$ there exists $K>0$ such that,  for all $u\in \mathcal{M}_{\mu}$, $\|u\|_{\e}\geq K$ and that for any $u\in H^{s}(\R^N,\R)\setminus \{0\}$ there exists a unique $t_{0}= t_{0}(u)$ such that $I_{\mu}(t_{0}u)= \max_{t\geq 0} I_{\mu}(tu)$ and then $t_{0}u\in \mathcal{M}_{\mu}$.
\end{Rem}

Using the same arguments of Lemma \ref{compactness} and arguing as in \cite[Lemma 6]{FigS} we can get
\begin{Lem}\label{lem4.3}
	Let $(w_{n})\subset \mathcal{M}_{\mu}$ be a sequence satisfying $I_{\mu}(w_{n})\rightarrow c_{\mu}$. Then $(w_{n})$ is bounded in $H^{s}(\R^{N},\R)$ and, up to a subsequence, $w_n\rightharpoonup w$ in $H^{s}(\R^{N},\R)$.  If $w\neq 0$, then $w_n\to w\in \mathcal{M}_{\mu}$ in $H^{s}(\R^{N},\R)$ and $w$ is a ground state for \eqref{Plim}.
	If $w=0$, then there exist $(\tilde{y}_{n})\subset \R^{N}$ and $\tilde{w}\in H^{s}(\R^{N},\R)\setminus\{0\}$ such that up to a subsequence $w_{n}(\cdot +\tilde{y}_{n})\to\tilde{w}\in \mathcal{M}_{\mu}$ in $H^{s}(\R^{N},\R)$ and $\tilde{w}$ is a ground state for \eqref{Plim}.
\end{Lem}

\begin{Rem}\label{remdecay}
In view of \cite[Theorems 1.2 and 3.4]{FQT} we can see that a ground state $\upsilon$ obtained in Lemma \ref{lem4.3} is H\"older continuous and has a power type decay at infinite, more precisely
	$$
	0<\upsilon(x)\leq \frac{C}{|x|^{N+2s}} \mbox{ if } |x|>1.
	$$
\end{Rem}

Now we prove a fundamental property on the $(PS)_{d}$ sequences for $J_{\e}$ in the {\em noncoercive} case ($V_{\infty}<\infty$).
\begin{Lem}\label{lem2.3}
	Let $d\in\R$. Assume that $V_{\infty}<\infty$ and let $(v_{n})$ be a $(PS)_{d}$ sequence for $J_{\e}$ in $\h$ with $v_{n}\rightharpoonup 0$ in $\h$. If $v_{n}\not \rightarrow 0$ in $\h$, then $d\geq c_{V_{\infty}}$.
\end{Lem}

\begin{proof}
	Let $(t_{n})\subset (0, +\infty)$ such that $(t_{n}|v_{n}|)\subset \mathcal{M}_{V_{\infty}}$.\\
	Firstly we prove that $	\limsup_{n} t_{n} \leq 1$.\\
	Assume by contradiction that there exist $\delta>0$ and a subsequence, still denoted by $(t_{n})$, such that 
	\begin{equation}\label{tv6}
	t_{n}\geq 1+ \delta \quad \forall n\in \mathbb{N}. 
	\end{equation}
	Since $(v_{n})$ is a $(PS)_{d}$ sequence for $J_{\e}$ we have 
	\begin{equation}\label{tv7}
	[v_{n}]_{A_{\e}}^{2} + \int_{\R^{N}} V(\e x) |v_{n}|^{2} dx = \int_{\R^{N}} f(|v_{n}|^{2}) |v_{n}|^{2} dx +o_{n}(1). 
	\end{equation}
	On the other hand, $t_{n}|v_{n}| \in \mathcal{M}_{V_{\infty}}$. Thus we get
	\begin{equation}\label{tv8}
	 [|v_{n}|]^{2}+  V_{\infty} |v_{n}|_2^{2}  =  \int_{\R^{N}} f(t^{2}_{n}|v_{n}|^{2}) |v_{n}|^{2} dx. 
	\end{equation}
	Putting together 
	\eqref{tv7}, \eqref{tv8} and using
	\eqref{eqDI} we obtain
	\begin{equation}\label{tv88}
	\int_{\R^{N}} \left[ f(t^{2}_{n}|v_{n}|^{2}) - f(|v_{n}|^{2}) \right]|v_{n}|^{2} \,dx\leq \int_{\R^{N}} \left( V_{\infty} - V(\e x)\right) |v_{n}|^{2} dx +o_{n}(1).
	\end{equation}
	Now, by the assumption \eqref{condV}, we can see that for every $\zeta>0$ there exists $R=R(\zeta)>0$ such that 
	\begin{equation}\label{tv9}
	V_{\infty} - V(\e x) \leq  \zeta \quad \mbox{ for any } |x|\geq R. 
	\end{equation}
	Combining \eqref{tv9} with the fact that, by Lemma \ref{embedding}, $v_{n}\rightarrow 0$ in $L^{2}(B_{R}, \C)$, so that $|v_{n}|\rightarrow 0$ in $L^{2}(B_{R})$, and with the boundedness of $(v_{n})$ in $\h$, we get
	\begin{align*}
	\int_{\R^{N}} \left( V_{\infty} - V(\e x)\right) |v_{n}|^{2} dx
	&= \int_{B_R(0)} \left( V_{\infty} - V(\e x)\right) |v_{n}|^{2} dx
	+ \int_{B_R^c(0)} \left( V_{\infty} - V(\e x)\right) |v_{n}|^{2} dx\\
	&\leq V_{\infty}\int_{B_R(0)} |v_{n}|^{2} dx + \zeta \int_{B_R^c(0)} |v_{n}|^{2} dx\\
	&\leq o_{n}(1) + \frac{\zeta}{V_{0}} \|v_{n}\|_{\e}^{2}\leq o_{n}(1)+ \zeta C. 
	\end{align*}
	Thus, in view of \eqref{tv88}, we deduce that
	\begin{equation}\label{tv10}
	\int_{\R^{N}} \left[ f(t^{2}_{n}|v_{n}|^{2}) - f(|v_{n}|^{2}) \right]|v_{n}|^{2} \,dx\leq \zeta C +o_{n}(1).
	\end{equation}
	Since $v_{n}\not \rightarrow 0$, we can apply Lemma \ref{compactness} to deduce the existence of a sequence $(y_{n})\subset \R^{N}$, and the existence of two positive numbers $\bar{R}, \beta$ such that
	\begin{equation}\label{tv11}
	\int_{B_{\bar{R}}(y_{n})} |v_{n}|^{2} dx \geq \beta>0.
	\end{equation}
	Now, let us consider $w_{n}= |v_{n}|(\cdot+y_{n})$. Taking into account that \eqref{condV}, \eqref{eqDI}, and the boundedness of $(v_{n})$ in $\h$, we can see that 
	\[
	\|w_{n}\|_{V_{0}}^{2}= \||v_{n}|\|_{V_{0}}^{2} \leq 
	\|v_{n}\|_{\e}^{2}\leq C.
	\]
	Therefore $w_{n}\rightharpoonup w$ in $H^{s}(\R^{N}, \R)$ and $w_{n}\rightarrow w$ in $L^{r}_{\rm loc}(\R^{N}, \R)$ for all $r\in [2, 2^{*}_{s})$. By \eqref{tv11} 
	\[
	\int_{B_{\bar{R}}(0)} w^2
	= \lim_n \int_{B_{\bar{R}}(0)} w_n^2
	\geq \beta
	\]
	and so there exists $\Omega \subset \R^{N}$ with positive measure and such that $w\neq 0$ in $\Omega$.
	By using \eqref{tv6}  and \eqref{tv10} we can infer
	\begin{align*}
	\int_{\Omega} \left(f((1+\delta)^2 w_{n}^{2})- f(w^{2}_{n})\right) w_{n}^{2}dx \leq \zeta C+o_{n}(1). 
	\end{align*}
	By applying Fatou's Lemma and by (\ref{f5}) we obtain
	\begin{align*}
	0<\int_{\Omega} \left(f((1+\delta)^2  w^{2})- f(w^{2})\right) w^{2}dx \leq \zeta C 
	\end{align*}
	and by the arbitrariness of $\zeta>0$ we get a contradiction. \\
	Now, two cases can occur.\\
	{\bf Case 1:} $\limsup_{n} t_{n}=1$.\\
	In this case there exists  a subsequence still denoted by $(t_{n})$ such that $t_{n}\rightarrow 1$. Taking into account that $\{v_{n}\}$ is a $(PS)_{d}$ sequence for $J_{\e}$, $c_{V_{\infty}}$ is the minimax level of $I_{V_{\infty}}$, and \eqref{eqDI}, we have
	\begin{equation}\label{tv12new}
	\begin{split}
	d+ o_{n}(1)&= J_{\e}(v_{n}) \\
	&\geq J_{\e}(v_{n}) -I_{V_{\infty}}(t_{n}|v_{n}|) + c_{V_{\infty}}\\
	&\geq \frac{1-t_{n}^{2}}{2} [|v_{n}|]^{2} + \frac{1}{2} \int_{\R^{N}} \left( V(\e x) - t_{n}^{2} V_{\infty}\right) |v_{n}|^{2} dx \\
	&\qquad
	+ \frac{1}{2}\int_{\R^{N}} \left[ F(t^{2}_{n} |v_{n}|^{2}) -F(|v_{n}|^{2}) \right] \, dx+ c_{V_{\infty}}.
	\end{split}
	\end{equation}
Since $(|v_{n}|)$ is bounded in $H^{s}(\R^{N}, \R)$ and $t_{n}\rightarrow 1$, we can see that 
\begin{equation}\label{tv14}
\frac{(1-t_{n}^{2})}{2} [|v_{n}|]^{2}= o_{n}(1). 
\end{equation}
	Now, using \eqref{condV}, we have that for every $\zeta>0$ there exists $R=R(\zeta)>0$ such that for any $|x|>R$ it holds
	\[
	V(\e x) - t_{n}^{2} V_{\infty} =\left(V(\e x) - V_{\infty} \right) + (1- t_{n}^{2}) V_{\infty}\geq -\zeta + (1- t_{n}^{2}) V_{\infty}.
	\]
	Thus, since $(v_{n})$ is bounded in $\h$, $|v_{n}|\rightarrow 0$ in $L^{p}(B_{R})$, $t_{n}\rightarrow 1$,  we get
	\begin{equation}\label{tv13}
	\begin{split}
	\int_{\R^{N}} \left( V(\e x) - t_{n}^{2} V_{\infty}\right) |v_{n}|^{2} dx
	&= \int_{B_R(0)} \left( V(\e x) - t_{n}^{2} V_{\infty}\right) |v_{n}|^{2} dx\\
	&\qquad
	+ \int_{B_R^c(0)} \left( V(\e x) - t_{n}^{2} V_{\infty}\right) |v_{n}|^{2} dx \\
	&\geq 
	(V_{0}- t_{n}^{2}V_{\infty}) \int_{B_R(0)} |v_{n}|^{2} dx 
	- \zeta \int_{B_R^c(0)} |v_{n}|^{2} dx\\
	&\qquad
	+ V_{\infty}(1- t_{n}^{2}) \int_{B_R^c(0)} |v_{n}|^{2} dx  \\
	&\geq o_{n}(1)- \frac{C}{V_0}\zeta . 
	\end{split}
	\end{equation}
	Finally, using the Mean Value Theorem, (\ref{propf1}) in Lemma \ref{propf}, $t_{n}\rightarrow 1$, and the boundedness of $(|v_{n}|)$, we get 
	\begin{equation}\label{tv16}
	\begin{split}
	\left|\int_{\R^{N}} \left[ F(t^{2}_{n} |v_{n}|^{2}) -F(|v_{n}|^{2}) \right] \, dx\right|
	&\leq
	\int_{\R^{N}} |f(\theta_n |v_n|^2)| |t^{2}_{n}-1| |v_{n}|^{2} \, dx\\
	&\leq
	(C_1 |v_{n}|_2^{2}+ C_2 |v_{n}|_q^{q})  |t^{2}_{n}-1| 
	=o_{n}(1).
	\end{split}
	\end{equation}
	Now, putting together \eqref{tv12new}, \eqref{tv14}, \eqref{tv13} and \eqref{tv16} we can infer that 
	\begin{align*}
	d+ o_{n}(1)\geq o_{n}(1) - \zeta C + c_{V_{\infty}}, 
	\end{align*}
	and taking the limit as $n\rightarrow \infty$ we get $d \geq c_{V_{\infty}}$. \\
	{\bf Case 2:} $\limsup_{n} t_{n}=t_{0}<1$. \\
	In this case there exists a subsequence still denoted by $(t_{n})$, such that $t_{n}\rightarrow t_{0}$ and $t_{n}<1$ for any $n\in \mathbb{N}$. 
	Since $(v_{n})$ is a  bounded $(PS)_{d}$ sequence for $J_{\e}$, we have 
	\begin{align}\label{tv17}
	d+o_{n}(1)= J_{\e}(v_{n}) - \frac{1}{2} \langle J'_{\e}(v_{n}), v_{n}\rangle  = \frac{1}{2} \int_{\R^{N}} \left(f(|v_{n}|^{2}) |v_{n}|^{2} -  F(|v_{n}|^{2})\right) \,dx. 
	\end{align}
	Observe that, by (\ref{f5}), the map $t\mapsto f(t)t- F(t)$ is increasing for $t>0$.\\
	Hence, since  $t_{n}|v_{n}|\in \mathcal{M}_{V_{\infty}}$ and $t_{n}<1$, from \eqref{tv17}, we obtain
	\begin{align*}
	c_{V_{\infty}} &\leq I_{V_{\infty}}(t_{n}|v_{n}|)  \\
	&= I_{V_{\infty}}(t_{n}|v_{n}|) -t_{n} \frac{1}{2} \langle I'_{V_{\infty}}(t_{n}|v_{n}|),|v_{n}|\rangle \\
	&=\frac{1}{2} \int_{\R^{N}}  \left( f(t^{2}_{n}|v_{n}|^{2}) t^{2}_{n}|v_{n}|^{2}- F(t^{2}_{n}|v_{n}|^{2}) \right) dx \\
	&\leq  \frac{1}{2} \int_{\R^{N}} \left( f(|v_{n}|^{2}) |v_{n}|^{2} - F(|v_{n}|^{2})\right) \,dx \\
	&=d +o_{n}(1). 
	\end{align*}
	Passing to the limit as $n\rightarrow \infty$ we get $d\geq c_{V_{\infty}}$.
\end{proof}

Thus we are ready to give conditions on the levels $c$ so that  $J_{\e}$ satisfies the $(PS)_{c}$ condition.

\begin{Prop}\label{prop2.1}
	The functional $J_{\e}$ satisfies the $(PS)_{c}$ condition at any level $c<c_{V_{\infty}}$ if $V_{\infty}<\infty$ and at any level $c\in \R$ if $V_{\infty}=\infty$. 
\end{Prop}

\begin{proof}
	Let $(u_{n})$ be a $(PS)_{c}$ sequence for $J_{\e}$. Then $(u_{n})$ is bounded in $\h$ and, up to a subsequence, $u_{n}\rightharpoonup u$ in $\h$ and $u_{n}\rightarrow u$ in $L^{q}_{\rm loc}(\R^{N}, \C)$ for any $q\in [1, 2^{*}_{s})$.
	Using also the assumptions (\ref{f2}), (\ref{f3}),
	it is easy to deduce that $J'_{\e}(u)=0$ and so, using (\ref{f4}), we can see that
	\begin{equation}\label{tv173}
	J_{\e}(u)=J_{\e}(u)-\frac{1}{2} \langle J'_{\e}(u),u\rangle=\frac{1}{2} \int_{\R^{N}} (f(|u|^{2})|u|^{2}- F(|u|^{2}))dx\geq 0.
	\end{equation}
	In view of Lemma \ref{vanishing} we can find a subsequence $(u_{n_{j}})\subset \h$ verifying \eqref{DL}.\\
	Now, let $v_{j}= u_{n_{j}}-\hat{u}_{j}$ where $\hat{u}_{j}$ is defined as in \eqref{eqtruncation}. 
	We claim that
	\begin{align}\label{tv171}
	J_{\e}(v_{j})= c-J_{\e}(u)+o_{j}(1)
	\end{align}
	and 
	\begin{align}\label{tv172}
	J'_{\e}(v_{j})=o_{j}(1).
	\end{align}
	To prove \eqref{tv171},	let us observe that
	\begin{equation*}
	\begin{split}
	J_{\e}(v_{j})-J_{\e}(u_{n_{j}})+J_{\e}(\hat{u}_{j})
	&\, =
	[\|\hat{u}_j\|_{\eps}^2 - \langle u_{n_{j}}, \hat{u}_{j}\rangle_\eps]
	+\int_{\R^{N}} [F(|u_{n_{j}}|^{2})-F(|v_{j}|^{2})-F(|\hat{u}_{j}|^{2})] dx \\
	&=:A_{j}+B_{j}. 
	\end{split}
	\end{equation*}
	In view of the weak convergence of $(u_{n_{j}})$ to $u$ in $\h$ and Lemma \ref{truncation}, we can see that $A_{j} \rightarrow 0$ as $j\rightarrow \infty$.
	Moreover, by (\ref{propf3}) in Lemma \ref{propf}, we have that  $B_{j}\rightarrow 0$  as $j\rightarrow \infty$.\\
	To show \eqref{tv172} we observe that
	\begin{align*}
	\left|\langle  J'_{\e}(v_{j})-J'_{\e}(u_{n_{j}})+J'_{\e}(\hat{u}_{j}), \phi \rangle \right| 
	&=\left|\Re\int_{\R^{N}} [f(|u_{n_{j}}|^{2})u_{n_{j}}-f(|v_{j}|^{2})v_{j}-f(|\hat{u}_{j}|^{2})\hat{u}_{j}] \bar{\phi} dx\right| \\
	&\leq \int_{\R^{N}} |f(|u_{n_{j}}|^{2})u_{n_{j}}-f(|v_{j}|^{2})v_{j}-f(|\hat{u}_{j}|^{2})\hat{u}_{j}| |\phi| dx
	\end{align*}
	and so, by (\ref{propf4}) in Lemma \ref{propf} we get that $\langle  J'_{\e}(v_{j})-J'_{\e}(u_{n_{j}})+J'_{\e}(\hat{u}_{j}), \phi \rangle\rightarrow 0$ for any $\phi\in \h$ such that $\|\phi\|_{\e}\leq 1$. Thus, since $J'_{\e}(u_{n_{j}})\rightarrow 0$ and $J'_{\e}(\hat{u}_{j})\rightarrow J'_{\e}(u)=0$, we can infer that \eqref{tv172} is satisfied.\\
	Let us assume that $V_{\infty}<\infty$ and $c<c_{V_\infty}$. By \eqref{tv171} and \eqref{tv173} we have that $c-J_{\e}(u)\leq c <c_{V_\infty}$.
		Thus, since $(v_j)$ is a $(PS)_{c-J_{\e}(u)}$ sequence for $J_{\e}$ and $v_{j}\rightharpoonup 0$ in $\h$, by Lemma \ref{lem2.3} we infer $v_{j}\rightarrow 0$ in $\h$. Hence Lemma \ref{truncation} implies that  $u_{n_{j}}\rightarrow u$ in $\h$ as $j\rightarrow \infty$.\\
	If $V_{\infty}=+\infty$. Then, by Lemma \ref{embedding},  $v_{j}\rightarrow 0$ in $L^{r}(\R^{N}, \C)$ for any $r\in [2, 2^{*}_{s})$ and by \eqref{tv172} and (\ref{propf1}) in Lemma \ref{propf} we deduce that 
	$$
	\|v_{j}\|^{2}_{\e}=\int_{\R^{N}} f(|v_{j}|^{2})|v_{j}|^{2}dx+ o_{j}(1)=o_{j}(1).
	$$
	Hence, as before, $u_{n_{j}}\rightarrow u$ in $\h$ as $j\rightarrow \infty$ and we conclude.
\end{proof}

Now we show that $\N_{\e}$ is a natural constraint, namely that the constrained critical points of the functional $J_{\e}$ on $\N_{\e}$ are critical points of $J_{\e}$ in $\h$.
\begin{Prop}\label{prop2.2}
	The functional $J_{\e}$ restricted to $\N_{\e}$ satisfies the $(PS)_{c}$ condition at any level $c<c_{V_{\infty}}$ if $V_{\infty}<\infty$ and at any level $c\in \R$ if $V_{\infty}=\infty$.
\end{Prop}
\begin{proof}
	Let $(u_{n})\subset \N_{\e}$ be a $(PS)_{c}$ sequence of restricted to $\N_{\e}$. Then, by \cite[Proposition 5.12]{W}, $J_{\e}(u_{n})\rightarrow c$  as $n\to\infty$ and there exists $(\lambda_{n})\subset \R$ such that	
	\begin{equation}\label{tv18}
	J'_{\e}(u_{n})=\lambda_{n} T'_{\e}(u_{n})+o_{n}(1)
	\end{equation}
	where $	T_{\e}: H^{s}_{\e}\rightarrow \R$ is defined as
	$$
	T_{\e}(u)=\|u\|^{2}_{\e}-\int_{\R^{N}} f(|u|^{2})|u|^{2} dx.
	$$
	By (\ref{f5}) we can see that
	\begin{align*}
	\langle T'_{\e}(u_{n}),u_{n}\rangle&=2\|u_{n}\|^{2}_{\e}-2\int_{\R^{N}} f(|u_{n}|^{2})|u_{n}|^{2}dx-2\int_{\R^{N}} f'(|u_{n}|^{2})|u_{n}|^{4}dx \\
	&=-2\int_{\R^{N}} f'(|u_{n}|^{2})|u_{n}|^{4}dx\leq -2C_{\sigma} |u_{n}|_\sigma^{\sigma}< 0.
	\end{align*}
	Up to a subsequence, we may assume that $\langle T'_{\e}(u_{n}),u_{n}\rangle\rightarrow \ell\leq 0$.\\
	If $\ell=0$, then $$o_{n}(1)=|\langle T'_{\e}(u_{n}),u_{n}\rangle|\geq C |u_{n}|_\sigma^{\sigma}$$ so we obtain that $u_{n}\rightarrow 0$ in $L^{\sigma}(\R^{N}, \C)$.
	Observe that, since $(u_{n})\subset \N_{\e}$ and $J_{\e}(u_{n})\rightarrow c$  as $n\to\infty$, then $(u_{n})$ is bounded in $\h$. Thus, by interpolation, we also have $u_{n}\rightarrow 0$ in $L^{q}(\R^{N}, \C)$. Hence, by (\ref{propf1}) in Lemma \ref{propf}, we get
	$$
	\|u_{n}\|^{2}_{\e}=\int_{\R^{N}} f(|u_{n}|^{2})|u_{n}|^{2}dx\leq \frac{\xi}{V_{0}} \|u_{n}\|^{2}_{\e}+C_{\xi}|u_{n}|_{q}^{q}= \frac{\xi}{V_{0}} \|u_{n}\|^{2}_{\e}+o_{n}(1),
	$$
	which implies that $u_{n}\rightarrow 0$ in $\h$. This is impossible in view of (\ref{2.1}) of Lemma \ref{LemNeharyE}. Therefore $\ell<0$ and by \eqref{tv18} we deduce that $\lambda_{n}=o_{n}(1)$. 
	Moreover, by the assumptions on $f$ we have that for every $\phi\in\h$
	\begin{align*}
	|\langle T'_{\e}(u_{n}),\phi \rangle|
	&\leq
	2\|u_{n}\|_{\e}\|\phi\|_{\e}
	+ 2\int_{\R^{N}} |f(|u_{n}|^{2})| |u_{n}| |\phi|dx
	+2\int_{\R^{N}} |f'(|u_{n}|^{2}) | |u_{n}|^{3} |\phi| dx \\
	&\leq
	C\|u_{n}\|_{\e} (1+ \|u_{n}\|_{\e}^{q-2})\|\phi\|_{\e}.
	\end{align*}
	Then, the boundedness of $(u_{n})$ implies the boundedness of $T'_{\e}(u_{n})$ and so, by \eqref{tv18} we infer that $J'_{\e}(u_{n})=o_{n}(1)$, that is $(u_{n})$ is a $(PS)_{c}$ sequence for $J_{\e}$. Hence, it is enough to apply Proposition \ref{prop2.1} to obtain the thesis.
\end{proof}

As a consequence we have the following result.
\begin{Cor}\label{cor2.1}
	The constrained critical points of the functional $J_{\e}$ on $\N_{\e}$ are critical points of $J_{\e}$ in $\h$.
\end{Cor}

Now we are ready the proof of the main result of this section.

\begin{proof}[Proof of Theorem \ref{thex}]
	By Lemma \ref{MPG} we know that $J_{\e}$ has a mountain pass geometry. So, by the Ekeland Variational Principle, there exists a $(PS)_{c_\eps}$ sequence $(u_{n})\subset \h$ for $J_{\e}$.\\
	If $V_{\infty}=\infty$, by Lemma \ref{embedding} and Proposition \ref{prop2.1} we deduce that $J_{\e}(u)=c_{\e}$ and $J'_{\e}(u)=0$, where $u\in \h$ is the weak limit of $u_{n}$.\\
	Now, we consider the case $V_{\infty}<\infty$. In view of Proposition \ref{prop2.1} it is enough to show that $c_{\e}<c_{V_{\infty}}$. Suppose without loss of generality that 
	$$
	V(0)=V_{0}=\inf_{x\in \R^{N}} V(x).
	$$
	Let $\mu\in  (V_{0}, V_{\infty})$. Clearly $c_{V_{0}}<c_{\mu}<c_{V_{\infty}}$. Let $w\in H^{s}(\R^{N}, \R)$ be a positive ground state to the autonomous problem \eqref{Plim} and $\eta\in C^{\infty}_{c}(\R^{N}, \R)$ be a cut-off function such that $\eta=1$ in $B_{1}(0)$ and $\eta=0$ in $B_{2}^{c}(0)$. Let us define $w_{r}(x):=\eta_{r}(x)w(x) e^{\imath A(0)\cdot x}$, with $\eta_r(x)=\eta(x/r)$ for $r>0$, and we observe that $|w_{r}|=\eta_{r}w$ and $w_{r}\in \h$ in view of Lemma \ref{aux}.
	Take $t_{r}>0$ such that 
	\begin{equation*}
	I_{\mu}(t_{r} |w_{r}|)=\max_{t\geq 0} I_{\mu}(t |w_{r}|)
	\end{equation*}
	Let us prove that there exists $r$ sufficiently large such that $I_{\mu}(t_{r}|w_{r}|)<c_{V_{\infty}}$.\\
	If by contradiction $I_{\mu}(t_{r}|w_{r}|)\geq c_{V_{\infty}}$ for any $r>0$, by using the fact that $|w_{r}|\rightarrow w$ in $H^{s}(\R^{N}, \R)$ as $r\rightarrow \infty$ (see \cite[Lemma 5]{PP}), we have $t_{r}\rightarrow 1$ and
	$$
	c_{V_{\infty}}\leq \liminf_{r\rightarrow \infty} I_{\mu}(t_{r}|w_{r}|)=I_{\mu}(w)=c_{\mu}
	$$
	which gives a contradiction since $c_{V_{\infty}}>c_{\mu}$.
	Hence, there exists $r>0$ such that
	\begin{align}\label{tv19}
	I_{\mu}(t_{r}|w_{r}|)=\max_{\tau\geq 0} I_{\mu}(\tau (t_{r} |w_{r}|)) \mbox{ and } I_{\mu}(t_{r}|w_{r}|)<c_{V_{\infty}}.
	\end{align}
	Now, we show that 
	\begin{equation}\label{limwr}
	\lim_{\e\rightarrow 0}[w_{r}]^{2}_{A_{\e}}=[\eta_{r}w]^{2}.
	\end{equation}
	Then we can see that
	\begin{align*}
	[w_{r}]_{A_\eps}^{2}
	&=\iint_{\R^{2N}} \frac{|e^{\imath A(0)\cdot x}\eta_{r}(x)w(x)-e^{\imath A_{\e}(\frac{x+y}{2})\cdot (x-y)}e^{\imath A(0)\cdot y} \eta_{r}(y)w(y)|^{2}}{|x-y|^{N+2s}} dx dy \nonumber \\
	&=[\eta_{r} w]^{2}
	+\iint_{\R^{2N}} \frac{\eta_{r}^2(y)w^2(y) |e^{\imath [A_{\e}(\frac{x+y}{2})-A(0)]\cdot (x-y)}-1|^{2}}{|x-y|^{N+2s}} dx dy\\
	&\quad+2\Re \iint_{\R^{2N}} \frac{(\eta_{r}(x)w(x)-\eta_{r}(y)w(y))\eta_{r}(y)w(y)(1-e^{-\imath [A_{\e}(\frac{x+y}{2})-A(0)]\cdot (x-y)})}{|x-y|^{N+2s}} dx dy \\
	&=: [\eta_{r} w]^{2}+X_{\e}+2Y_{\e}.
	\end{align*}
	Since 
	$|Y_{\e}|\leq [\eta_{r} w] \sqrt{X_{\e}}$, it s enough to show that	$X_{\e}\rightarrow 0$ as $\e\rightarrow 0$ to deduce that \eqref{limwr} holds.\\
	Observe that, for $0<\beta<\alpha/({1+\alpha-s})$, 
	\begin{equation}\label{Ye}
	\begin{split}
	X_{\e}
	&\leq \int_{\R^{N}} w^{2}(y) dy \int_{|x-y|\geq\e^{-\beta}} \frac{|e^{\imath [A_{\e}(\frac{x+y}{2})-A(0)]\cdot (x-y)}-1|^{2}}{|x-y|^{N+2s}} dx\\
	&+\int_{\R^{N}} w^{2}(y) dy  \int_{|x-y|<\e^{-\beta}} \frac{|e^{\imath [A_{\e}(\frac{x+y}{2})-A(0)]\cdot (x-y)}-1|^{2}}{|x-y|^{N+2s}} dx \\
	&=:X^{1}_{\e}+X^{2}_{\e}.
	\end{split}
	\end{equation}
	Since $|e^{\imath t}-1|^{2}\leq 4$ and recalling that $w\in H^{s}(\R^{N}, \R)$, we can observe that 
	\begin{equation}\label{Ye1}
	X_{\e}^{1}\leq C \int_{\R^{N}} w^{2}(y) dy \int_{\e^{-\beta}}^\infty \rho^{-1-2s} d\rho\leq C \e^{2\beta s} \rightarrow 0.
	\end{equation}
	Concerning $X^{2}_{\e}$, since $|e^{\imath t}-1|^{2}\leq t^{2}$ for all $t\in \R$, $A\in C^{0,\alpha}(\R^N,\R^N)$ for $\alpha\in(0,1]$, and $|x+y|^{2}\leq 2(|x-y|^{2}+4|y|^{2})$,	 we have
	\begin{equation}\label{Ye2}
	\begin{split}
	X^{2}_{\e}&
	\leq \int_{\R^{N}} w^{2}(y) dy  \int_{|x-y|<\e^{-\beta}} \frac{|A_{\e}\left(\frac{x+y}{2}\right)-A(0)|^{2} }{|x-y|^{N+2s-2}} dx \\
	&\leq C\e^{2\alpha} \int_{\R^{N}} w^{2}(y) dy  \int_{|x-y|<\e^{-\beta}} \frac{|x+y|^{2\alpha} }{|x-y|^{N+2s-2}} dx \\
	&\leq C\e^{2\alpha} \left(\int_{\R^{N}} w^{2}(y) dy  \int_{|x-y|<\e^{-\beta}} \frac{1 }{|x-y|^{N+2s-2-2\alpha}} dx\right.\\
	&\qquad\qquad+ \left. \int_{\R^{N}} |y|^{2\alpha} w^{2}(y) dy  \int_{|x-y|<\e^{-\beta}} \frac{1}{|x-y|^{N+2s-2}} dx\right) \\
	&=: C\e^{2\alpha} (X^{2, 1}_{\e}+X^{2, 2}_{\e}).
	\end{split}
	\end{equation}	
	Then
	\begin{equation}\label{Ye21}
	X^{2, 1}_{\e}
	= C  \int_{\R^{N}} w^{2}(y) dy \int_0^{\e^{-\beta}} \rho^{1+2\alpha-2s} d\rho
	\leq C\eps^{-2\beta(1+\alpha-s)}.
	\end{equation}
	On the other hand, using Remark \ref{remdecay}, we infer that
	\begin{equation}\label{Ye22}
	\begin{split}
	 X^{2, 2}_{\e}
	 &\leq C  \int_{\R^{N}} |y|^{2\alpha} w^{2}(y) dy \int_0^{\e^{-\beta}}\rho^{1-2s} d\rho  \\
	&\leq C \e^{-2\beta(1-s)} \left[\int_{B_1(0)}  w^{2}(y) dy + \int_{B_1^c(0)} \frac{1}{|y|^{2(N+2s)-2\alpha}} dy \right]  \\
	&\leq C \e^{-2\beta(1-s)}.
	\end{split}
	\end{equation}
	Taking into account \eqref{Ye}, \eqref{Ye1}, \eqref{Ye2}, \eqref{Ye21} and \eqref{Ye22} we can conclude that $X_{\e}\rightarrow 0$.\\
	Now, in view of \eqref{condV}, there exists  $\e_{0}>0$ such that 
	\begin{equation}\label{tv20}
	V(\e x)\leq \mu \mbox{ for all } x\in \supp(|w_{r}|), \e\in (0, \e_{0}).
	\end{equation} 
	Therefore, putting together \eqref{tv19} , \eqref{limwr} and \eqref{tv20}, 
	we deduce that
	$$
	\limsup_{\e\rightarrow 0}c_{\e}\leq \limsup_{\e\rightarrow 0}\left[\max_{\tau\geq 0} J_{\e}(\tau t_{r} w_{r})\right]\leq \max_{\tau\geq 0} I_{\mu}(\tau t_{r} |w_{r}|)=I_{\mu}(t_{r}|w_{r}|)<c_{V_{\infty}}
	$$ 
	which implies that $c_{\e}<c_{V_{\infty}}$ for any $\e>0$ sufficiently small.
\end{proof}

\section{Proof of Theorem \ref{thmf}}\label{sec4}
In this section, our main purpose is to apply the Ljusternik-Schnirelmann category theory to prove a multiplicity result for problem \eqref{Pe}. In order to achieve our main result, first we give some useful preliminary lemmas.\\
Let $\delta>0$ be fixed and $\omega\in H^s(\R^N, \R)$ be a ground state solution of the problem \eqref{Plim} for $\mu=V_0$ given by Lemma \ref{lem4.3} (see also Remark \ref{remdecay}).

Moreover let $\psi\in C^{\infty}(\R^{+}, [0, 1])$ be a nonincreasing function such that $\psi=1$ in $[0, \delta/2]$ and $\psi=0$ in $[\delta, \infty)$ and, for any fixed $y\in M$, let us introduce
$$
\Psi_{\e, y}(x):=\psi(|\e x-y|) \omega\left(\frac{\e x-y}{\e}\right)e^{\imath \tau_{y}(\frac{\e x-y}{\e})}
$$
where $M$ is defined in \eqref{defM} and $\tau_{y}(x):=\sum_{j=1}^{N} A_{j}(y)x_{j}$.\\
By Lemma \ref{LemNeharyE} let $t_{\e}>0$ be the unique positive number such that
$$
J_{\e}(t_{\e}\Psi_{\e, y})=\max_{t\geq 0} J_{\e}(t_{\e}\Psi_{\e, y})
$$
and let us introduce the map $\Phi_{\e}:M\rightarrow \N_{\e}$ by setting $\Phi_{\e}(y)=t_{\e} \Psi_{\e, y}$. By construction, $\Phi_{\e}(y)$ has compact support for any $y\in M$.\\
We begin proving the following result.
\begin{Lem}\label{ADlem}
As $\eps\to 0$ we have that $\|\Psi_{\e, y}\|_\eps^2 \to \|\omega\|_{V_0}^2$ uniformly with respect to $y\in M$. 
\end{Lem}

\begin{proof}
By applying the Dominated Convergence Theorem we easily have that
\[
\int_{\R^{N}} V(\eps x) |\Psi_{\e, y}(x)|^2dx\to V_{0} \int_{\R^{N}}  \omega^2(x) dx.
\]
Thus, we only need to prove that as $\e\rightarrow 0$
\begin{equation*}
\iint_{\R^{2N}} \frac{|\Psi_{\e, y}(x_{1})-\Psi_{\e, y}(x_{2})e^{\imath (x_{1}-x_{2})\cdot A_{\e}(\frac{x_{1}+x_{2}}{2})}|^{2}}{|x_{1}-x_{2}|^{N+2s}}dx_{1}dx_{2}\rightarrow \iint_{\R^{2N}} \frac{|\omega(x_{1})-\omega(x_{2})|^{2}}{|x_{1}-x_{2}|^{N+2s}} dx_{1}dx_{2}.
\end{equation*}
By using the change of variable $\e x_{i}-y=\e z_{i}$ $(i=1, 2)$, we obtain
	\begin{align*}
	&\iint_{\R^{2N}} \frac{|\Psi_{\e, y}(x_{1})-\Psi_{\e, y}(x_{2})e^{\imath (x_{1}-x_{2})\cdot A_{\e}(\frac{x_{1}+x_{2}}{2})}|^{2}}{|x_{1}-x_{2}|^{N+2s}}dx_{1}dx_{2} \\
	&=\iint_{\R^{2N}} \frac{|\psi(|\e z_{1}|)\omega(z_{1})e^{\imath \tau_{y}(z_{1})}-\psi(|\e z_{2}|) \omega(z_{2}) e^{\imath \tau_{y}(z_{2})} e^{\imath (z_{1}-z_{2})\cdot A(\e \frac{z_{1}+z_{2}}{2}+y)}|^{2}}{|z_{1}-z_{2}|^{N+2s}}dz_{1}dz_{2} \\
	&=\iint_{\R^{2N}} \frac{|\psi(|\e z_{1}|)\omega(z_{1})-\psi(|\e z_{2}|)\omega(z_{2})|^{2}}{|z_{1}-z_{2}|^{N+2s}} dz_{1}dz_{2} \\
	&+2\iint_{\R^{2N}} \frac{\psi^2(|\e z_{2}|)\omega^2(z_{2}) \Bigl( 1- \cos \left\{(z_{1}-z_{2})\cdot  [A(\e(\frac{z_{1}+z_{2}}{2})+y)-A(y)] \right\}  \Bigr) }{|z_{1}-z_{2}|^{N+2s}} dz_{1}dz_{2} \\
	&+ 2\Re \iint_{\R^{2N}} \frac{[\psi(|\e z_{1}|)\omega(z_{1})-\psi(|\e z_{2}|)\omega(z_{2})]  \psi(|\e z_{2}|)\omega(z_{2}) \left[1-e^{\imath (z_{2}-z_{1})\cdot [A(\e(\frac{z_{1}+z_{2}}{2})+y)-A(y)]} \right]}{|z_{1}-z_{2}|^{N+2s}} dz_{1}dz_{2}\\
	&
	:=X_{\e}+Y_{\e}+2Z_{\e}\\
	\end{align*}
	Since $\psi(|x|)=1$ for $x\in B_{\delta/2}$, we can use \cite[Lemma $5$]{PP} to get
	$$
	X_{\e}=\iint_{\R^{2N}} \frac{|\psi(|\e z_{1}|)\omega(z_{1})-\psi(|\e z_{2}|)\omega(z_{2})|^{2}}{|z_{1}-z_{2}|^{N+2s}} dz_{1}dz_{2}\rightarrow \iint_{\R^{2N}} \frac{|\omega(z_{1})-\omega(z_{2})|^{2}}{|z_{1}-z_{2}|^{N+2s}} dz_{1}dz_{2}
	$$
	as $\e\rightarrow 0$.\\
	On the other hand,
by the H\"older inequality we can see that
	$$
	|Z_{\e}|\leq \sqrt{X_{\e}} \sqrt{Y_{\e}}.
	$$
	Therefore, it is enough to show that $Y_{\e}\rightarrow 0$ as $\e\rightarrow 0$.\\
%
	Being $\psi=0$ in $B_{\delta}^c(0)$, we have
	\begin{equation}\label{Be}
	\begin{split}
	Y_{\e}
	&=2\int_{B_{\delta/\eps}(0)} \psi^{2}(|\e z_{2}|) \omega^{2}(z_{2}) dz_{2}  \left\{ \int_{|z_{1}-z_{2}|<\e^{-\beta}} \frac{{1-\cos \left\{(z_{1}-z_{2})\cdot  [A(\e(\frac{z_{1}+z_{2}}{2})+y)-A(y)]\right\} }}{|z_{1}-z_{2}|^{N+2s}}  dz_{1}\right.\\
	&+\left.\int_{|z_{1}-z_{2}|\geq\e^{-\beta}} \frac{{1-\cos \left\{(z_{1}-z_{2})\cdot  [A(\e(\frac{z_{1}+z_{2}}{2})+y)-A(y)]\right\} }}{|z_{1}-z_{2}|^{N+2s}}  dz_{1}\right\}:=Y_{\e}^{1}+Y_{\e}^{2},
	\end{split}
	\end{equation}
	where $0<\beta<\frac{\alpha}{1+\alpha-s}$.\\
	Taking into account that $|z_{1}+z_{2}|^{2\alpha}\leq C(|z_{1}-z_{2}|^{2\alpha}+|z_{2}|^{2\alpha})$ for any $z_{1}, z_{2}\in \R^N$, $2(1-\cos t)\leq t^2$ in $\R$, the assumptions on $A$, and recalling that $0\leq \psi\leq 1$ we can see that 
\begin{equation}\label{Be1}
\begin{split}
	Y_{\e}^{1}
&\leq C\e^{2\alpha} \int_{B_{\delta/\eps}(0)} \omega^{2}(z_{2}) dz_{2}  \Bigl\{\int_{|z_{1}-z_{2}|<\e^{-\beta}} \frac{dz_{1}}{|z_{1}-z_{2}|^{N+2s-2-2\alpha}} 
+\int_{|z_{1}-z_{2}|<\e^{-\beta}} \frac{|z_{2}|^{2\alpha}}{|z_{1}-z_{2}|^{N+2s-2}}  dz_{1} \Bigr\}\\
&=: C\e^{2\alpha} [Y_{\e}^{1,1}+Y_{\e}^{1, 2}].
\end{split}
\end{equation}
	We have
	\begin{align}\label{Be2}
	Y_{\e}^{1, 1}
	&\leq C
	 \int_{\R^{N}} \omega^{2}(z_{2}) dz_{2} \int_{0}^{\e^{-\beta}} \rho^{1+2\alpha - 2s} d\rho
	 = C \e^{-2\beta(1+\alpha-s)} 
	\end{align}
	and, taking into account  Remark \ref{remdecay} and that $N\geq 3$,
\begin{equation}\label{Be3}
\begin{split}
Y_{\e}^{1, 2}
&\leq C \int_{\R^{N}} |z_{2}|^{2\alpha}\omega^2(z_{2})  dz_{2} \int_0^{\e^{-\beta}} \rho^{1-2s} d\rho\\
&\leq C \e^{-2\beta(1-s)} \left[\int_{|z_{2}|>1} \frac{1}{|z_{2}|^{2(N+2s)-2\alpha}}dz_{2} +\int_{|z_{2}|<1} \omega(z_{2})^{2}dz_{2}\right]\\
&\leq C \e^{-2\beta(1-s)}
\end{split}
\end{equation}
	Putting together \eqref{Be1}, \eqref{Be2} and \eqref{Be3} we can infer that 
	\begin{equation}\label{Be4}
	Y^{1}_{\e}\rightarrow 0 \mbox{ as } \e\rightarrow 0.
	\end{equation}
	Finally, using the facts $0\leq \psi\leq 1$ and $0\leq 1-\cos t \leq 1$ in $\R$, we have
	\begin{equation}\label{Be5}
	Y_{\e}^{2}
	\leq  C\int_{\R^{N}} \omega^2(z_{2}) dz_{2} \int_{\e^{-\beta}}^{\infty} \frac{1}{\rho^{2s+1}} d\rho
	\leq C \e^{2s\beta}.
	\end{equation}
Taking into account \eqref{Be},\eqref{Be4} and \eqref{Be5} we can conclude.
\end{proof}

The next result will be very useful to define a map from $M$ to a suitable sub level in the Nehari  manifold.
\begin{Lem}\label{lem4.1}
	The functional $\Phi_{\e}$ satisfies the following limit
	\begin{equation*}
	\lim_{\e\rightarrow 0} J_{\e}(\Phi_{\e}(y))=c_{V_{0}} \mbox{ uniformly in } y\in M.
	\end{equation*}
\end{Lem}
\begin{proof}
	Assume by contradiction that there there exists $\kappa>0$, $(y_{n})\subset M$ and $\e_{n}\rightarrow 0$ such that 
	\begin{equation*}
	|J_{\e_{n}}(\Phi_{\e_{n}}(y_{n}))-c_{V_{0}}|\geq \kappa.
	\end{equation*}
	Since $\langle J'_{\e_{n}}(\Phi_{\e_{n}}(y_{n})), \Phi_{\e_{n}}(y_{n})\rangle=0$ and using the change of variable $z=(\e_{n}x-y_{n})/{\e_{n}}$, (\ref{f5}), and that, if $z\in B_{\delta/\e_{n}}(0)$, then $\e_{n} z+y_{n}\in B_{\delta}(y_{n})\subset M_{\delta}$, we can see that
	\begin{equation*}
	\begin{split}
	\|\Psi_{\e_{n}, y_{n}}\|^{2}_{\e_{n}}
	&=\int_{\R^{N}} f(|t_{\e_{n}}\Psi_{\e_{n}}|^{2}) |\Psi_{\e_{n}}|^{2} dx\\
	&=\int_{\R^{N}} f(|t_{\e_{n}} \psi(|\e_{n}z|) \omega(z)|^{2}) |\psi(|\e_{n}z|) \omega(z)|^{2}  dz\\
	&\geq \int_{B_{{\delta}/{2}}(0)} f(|t_{\e_{n}} \omega(z)|^{2}) \omega^{2}(z)  dz\\
	&\geq f(|t_{n}\alpha|^{2})  \int_{B_{{\delta}/{2}}(0)}\omega^{2}(z)  dz
	\end{split}
	\end{equation*}
	for all $n\geq n_{0}$, with  $n_{0}\in \mathbb{N}$ such that $B_{\frac{\delta}{2}}(0)\subset B_{\frac{\delta}{2\e_{n}}}(0)$ and $\alpha=\min\{\omega(z): |z|\leq \frac{\delta}{2}\}$.\\
	Hence, if $t_{\e_{n}}\rightarrow \infty$, by (\ref{f4})
	we deduce that $\|\Psi_{\e_{n}, y_{n}}\|^{2}\rightarrow \infty$ which contradicts Lemma \ref{ADlem}.\\
	Therefore, up to a subsequence, we may assume that $t_{\e_{n}}\rightarrow t_{0}\geq 0$.
	In fact, taking into account Lemma \ref{ADlem} and passing to the limit as $n\to\infty$ in
	\[
	\|\Psi_{\e_{n}, y_{n}}\|^{2}_{\e_{n}}
	=\int_{\R^{N}} f(|t_{\e_{n}} \psi(|\e_{n}z|) \omega(z)|^{2}) |\psi(|\e_{n}z|) \omega(z)|^{2}  dz
	\]
	it is easy to check that $t_{0}>0$.\\
	Moreover
	$$
	[t_{0}\omega]^{2}+\int_{\R^{N}} V_{0}|t_{0}\omega|^{2} dx=\int_{\R^{N}} f(|t_{0}\omega|^{2}) t_{0}^2\omega^{2},
	$$ 
	that is $t_{0}\omega\in \mathcal{M}_{V_0}$. Since $\omega\in  \mathcal{M}_{V_0}$ we get that $t_{0}=1$.\\
	Then
	\begin{equation*}
	\lim_{n}\int_{\R^{N}} F(|\Phi_{\e_{n}}(y_{n})|^{2})=\int_{\R^{N}} F(\omega^{2}).
	\end{equation*}
	and so
	$$
	\lim_{n} J_{\e_{n}}(\Phi_{\e_{n}}(y_{n}))=I_{V_{0}}(\omega)=c_{V_{0}}
	$$
	which gives a contradiction.
\end{proof}

Now, we are in the position to define the barycenter map. We take $\rho>0$ such that $M_{\delta}\subset B_{\rho}$ and we consider $\Upsilon: \R^{N}\rightarrow \R^{N}$ defined by setting
\begin{equation*}
\Upsilon(x)=
\left\{
\begin{array}{ll}
x &\mbox{ if } |x|<\rho \\
\rho x/{|x|} &\mbox{ if } |x|\geq \rho.
\end{array}
\right.
\end{equation*}
We define the barycenter map $\beta_{\e}: \N_{\e}\rightarrow \R^{N}$ as follows
\[
\beta_{\e}(u):=\frac{\displaystyle\int_{\R^{N}} \Upsilon(\e x) |u(x)|^{2} dx}{\displaystyle\int_{\R^{N}} |u(x)|^{2} dx}.
\]

\begin{Lem}\label{lem4.2}
	The function $\Phi_{\e}$ verifies the following limit
	\begin{equation*}
	\lim_{\e \rightarrow 0} \beta_{\e}(\Phi_{\e}(y))=y \mbox{ uniformly in } y\in M.
	\end{equation*}
\end{Lem}
\begin{proof}
	Suppose by contradiction that there exists $\kappa>0$, $(y_{n})\subset M$ and $\e_{n}\rightarrow 0$ such that 
	\begin{equation}\label{4.4}
	|\beta_{\e_{n}}(\Phi_{\e_{n}}(y_{n}))-y_{n}|\geq \kappa.
	\end{equation}
	Using the change of variable $z= ({\e_{n} x-y_{n}})/{\e_{n}}$, we can see that 
	$$
	\beta_{\e_{n}}(\Psi_{\e_{n}}(y_{n}))=y_{n}+\frac{\int_{\R^{N}}[\Upsilon(\e_{n}z+y_{n})-y_{n}] |\psi(|\e_{n}z|)|^{2} |\omega(z)|^{2} \, dz}{\int_{\R^{N}} |\psi(|\e_{n}z|)|^{2} |\omega(z)|^{2}\, dz}.
	$$
	Taking into account $(y_{n})\subset M\subset M_\delta \subset B_{\rho}$ and the Dominated Convergence Theorem, we can infer that 
	$$
	|\beta_{\e_{n}}(\Phi_{\e_{n}}(y_{n}))-y_{n}|=o_{n}(1)
	$$
	which contradicts (\ref{4.4}).
\end{proof}

Next, we prove the following useful compactness result.
\begin{Prop}\label{prop4.1}
	Let $\e_{n}\rightarrow 0^{+}$ and $(u_{n})\subset \N_{\e_{n}}$ be such that $J_{\e_{n}}(u_{n})\rightarrow c_{V_{0}}$. Then there exists $(\tilde{y}_{n})\subset \R^{N}$ such that the translated sequence 
	\begin{equation*}
	v_{n}(x):=|u_{n}|(x+ \tilde{y}_{n})
	\end{equation*}
	has a subsequence which converges in $H^{s}(\R^{N}, \R)$. Moreover, up to a subsequence, $(y_{n}):=(\e_{n}\tilde{y}_{n})$ is such that $y_{n}\rightarrow y\in M$. 
\end{Prop}

\begin{proof}
	Since $\langle J'_{\e_{n}}(u_{n}), u_{n} \rangle=0$ and $J_{\e_{n}}(u_{n})\rightarrow c_{V_{0}}$, we easily get that there exists $C>0$ such that $\|u_{n}\|_{\e_{n}}\leq C$ for all $n\in\mathbb{N}$.
	Let us observe that $\|u_{n}\|_{\e_{n}}\nrightarrow 0$ since $c_{V_{0}}>0$.
	Therefore, as in the proof of Lemma \ref{compactness}, we can find a sequence $(\tilde{y}_{n})\subset \R^{N}$ and constants $R, \beta>0$ such that
	\begin{equation}\label{tv21}
	\liminf_{n}\int_{B_{R}(\tilde{y}_{n})} |u_{n}|^{2} dx\geq \beta.
	\end{equation}
	Let us define 
	\begin{equation*}
	v_{n}(x):=|u_{n}|(x+ \tilde{y}_{n}). 
	\end{equation*}
	By the diamagnetic inequality  \eqref{eqDI} we get the boundedness of $(|u_{n}|)$ in $H^{s}(\R^{N}, \R)$ and, using  \eqref{tv21}, we may suppose that $v_{n}\rightharpoonup v$ in $H^{s}(\R^{N}, \R)$ for some $v\neq 0$.\\
	Let $(t_{n})\subset (0, +\infty)$ be such that $w_{n}=t_{n} v_{n}\in \mathcal{M}_{V_{0}}$, and set $y_{n}:=\e_{n}\tilde{y}_{n}$.  \\
	By
	\eqref{eqDI}, we can see that
	\begin{align*}
	c_{V_{0}}\leq I_{V_{0}}(w_{n})\leq \max_{t\geq 0} J_{\e_{n}}(t u_{n})=J_{\e_{n}}(u_{n})=c_{V_{0}}+ o_{n}(1), 
	\end{align*}
	which yields $I_{V_{0}}(w_{n})\rightarrow c_{V_{0}}$. \\
	Now, the sequence $(t_{n})$ is bounded since $(v_{n})$ and $(w_{n})$, by Lemma \ref{lem4.3},
	are bounded in $H^{s}(\R^{N}, \R)$ and $v_{n}\nrightarrow 0$ in $H^{s}(\R^{N}, \R)$. Therefore, up to a subsequence, we may assume that $t_{n}\rightarrow t_{0}\geq 0$.\\
	Let us show that $t_{0}>0$.\\
	In fact, if $t_{0}=0$, from the boundedness of $(v_{n})$, we get $w_{n}= t_{n}v_{n} \rightarrow 0$ in $H^{s}(\R^{N}, \R)$, that is $I_{V_{0}}(w_{n})\rightarrow 0$ in contrast with the fact $c_{V_{0}}>0$.\\
	Thus, up to a subsequence, we may assume that $w_{n}\rightharpoonup w:= t_{0} v\neq 0$ in $H^{s}(\R^{N}, \R)$. \\
	From Lemma \ref{lem4.3}, we can deduce that $w_{n} \rightarrow w$ in $H^{s}(\R^{N}, \R)$, which gives $v_{n}\rightarrow v$ in $H^{s}(\R^{N}, \R)$.\\
	Now show that $(y_{n})$ has a subsequence such that $y_{n}\rightarrow y\in M$.\\
	Assume by contradiction that $(y_{n})$ is not bounded, that is there exists a subsequence, still denoted by $(y_{n})$, such that $|y_{n}|\rightarrow +\infty$. \\
	Firstly, we deal with the case $V_{\infty}=\infty$. \\
	Taking into account \eqref{eqDI}, we can see that
	\[
	\int_{\R^{N}} V(\e_{n}x+y_{n})|v_{n}|^{2} dx\leq [|v_{n}|]^{2}+\int_{\R^{N}} V(\e_{n}x+y_{n})|v_{n}|^{2} dx
	\leq \|u_{n}\|^{2}_{\e_{n}}\leq C.
	\]
	On the other hand, by Fatou's Lemma, we deduce that
	\[
	\liminf_{n} \int_{\R^{N}} V(\e_{n}x+y_{n})|v_{n}|^{2}dx=\infty
	\]
	and we get a contradiction.\\
	Now, let us consider the case $V_{\infty}<\infty$. \\
	Since $w_{n}\rightarrow w$ strongly in $H^{s}(\R^{N}, \R)$, $V_{0}<V_{\infty}$, and by using \eqref{eqDI}, we obtain
	\begin{equation}\label{tv22}
	\begin{split}
	c_{V_{0}}&= I_{V_{0}}(w) < I_{V_{\infty}} (w)\\
	&\leq \liminf_{n} \left[\frac{1}{2} [w_{n}]^{2}+\frac{1}{2} \int_{\R^{N}} V(\e_{n} x+y_{n})|w_{n}|^{2} dx-\int_{\R^{N}} F(|w_{n}|^{2}) dx\right] \\
	&=\liminf_{n} \left[\frac{t^{2}_{n}}{2} [|u_{n}|]^{2}+\frac{t^{2}_{n}}{2} \int_{\R^{N}} V(\e_{n} z)|u_{n}|^{2} dx-\int_{\R^{N}} F(t_{n}^{2} |u_{n}|^{2}) dx\right] \\
	&\leq \liminf_{n} J_{\e_{n}}(t_{n}u_{n}) \leq \liminf_{n} J_{\e_{n}} (u_{n})=c_{V_{0}}
	\end{split}
	\end{equation}
	which gives a contradiction. \\
	Thus $(y_{n})$ is bounded and, up to a subsequence, we may assume that $y_{n}\rightarrow y$. If $y\notin M$, then $V_{0}<V(y)$ and we can argue as in \eqref{tv22} to get a contradiction and so the proof is complete.	
\end{proof}

\noindent
At this point, we introduce a subset $\widetilde{\N}_{\e}$ of $\N_{\e}$ by setting 
$$
\widetilde{\N}_{\e}=\{u\in \N_{\e}: J_{\e}(u)\leq c_{V_{0}}+h(\e)\},
$$
where $h:\R_{+}\rightarrow \R_{+}$ is such that $h(\varsigma)\rightarrow 0$ as $\varsigma\rightarrow 0$.\\
Fixed $y\in M$, we conclude from Lemma \ref{lem4.1} that $h(\varsigma)=|J_{\varsigma}(\Phi_{\varsigma}(y))-c_{V_{0}}|\rightarrow 0$ as $\varsigma \rightarrow 0$. Hence $\Phi_{\e}(y)\in \widetilde{\N}_{\e}$, and $\widetilde{\N}_{\e}\neq \emptyset$ for any $\e>0$.

Moreover, we have the following relation between $\widetilde{\N}_{\e}$ and the barycenter map.

\begin{Lem}\label{lem4.4}We have
	$$
	\lim_{\e \rightarrow 0} \sup_{u\in \widetilde{\mathcal{N}}_{\e}} \operatorname{dist}(\beta_{\e}(u), M_{\delta})=0.
	$$
\end{Lem}

\begin{proof}
	Let $\e_{n}\rightarrow 0$ as $n\rightarrow \infty$. For any $n\in \mathbb{N}$, there exists $(u_{n})\in \widetilde{\N}_{\e_{n}}$ such that
	$$
	\sup_{u\in \widetilde{\N}_{\e_{n}}} \inf_{y\in M_{\delta}}|\beta_{\e_{n}}(u)-y|=\inf_{y\in M_{\delta}}|\beta_{\e_{n}}(u_{n})-y|+o_{n}(1).
	$$
	Therefore, it is suffices to prove that there exists $(y_{n})\subset M_{\delta}$ such that 
	\begin{equation}\label{3.13}
	\lim_{n} |\beta_{\e_{n}}(u_{n})-y_{n}|=0.
	\end{equation}
	By using the diamagnetic inequality \eqref{eqDI}, we can see that $I_{V_{0}}(t|u_{n}|)\leq J_{\e_{n}}(t u_{n})$ for any $t\geq 0$. Therefore, recalling that $(u_{n})\subset  \widetilde{\N}_{\e_{n}}\subset  \N_{\e_{n}}$, we can deduce that
	$$
	c_{V_{0}}\leq \max_{t\geq 0} I_{V_{0}}(t|u_{n}|)\leq \max_{t\geq 0} J_{\e_{n}}(t u_{n})=J_{\e_{n}}(u_{n})\leq c_{V_{0}}+h(\e_{n})
	$$
	which implies that $J_{\e_{n}}(u_{n})\rightarrow c_{V_{0}}$ because of $h(\e_{n})\rightarrow 0$ as $n\rightarrow \infty$. \\
	From Proposition \ref{prop4.1} it follows that there exists $(\tilde{y}_{n})\subset \R^{N}$ such that $y_{n}=\e_{n}\tilde{y}_{n}\in M_{\delta}$ for $n$ sufficiently large. \\
	Thus
	$$
	\beta_{\e_{n}}(u_{n})=y_{n}+\frac{\int_{\R^{N}}[\Upsilon(\e_{n}z+y_{n})-y_{n}] |u_{n}(z+\tilde{y}_{n})|^{2} \, dz}{\int_{\R^{N}} |u_{n}(z+\tilde{y}_{n})|^{2} \, dz}.
	$$
	Since, up to a subsequence, $|u_{n}|(\cdot+\tilde{y}_{n})$ converges strongly in $H^{s}(\R^{N}, \R)$ and $\e_{n}z+y_{n}\rightarrow y\in M$ for any $z\in\R^N$, we deduce (\ref{3.13}).
\end{proof}

Now, we are ready to present the proof of our multiplicity result.

\begin{proof}[Proof of Theorem \ref{thmf}]
	Given $\delta>0$, we can apply Lemma \ref{lem4.1}, Lemma \ref{lem4.2} and Lemma \ref{lem4.4} and argue as in  \cite[Section $6$]{CL} to find $\e_{\delta}>0$ such that for any $\e\in (0, \e_{\delta})$, the diagram
	$$
	M \stackrel{\Phi_{\e}}{\rightarrow}  \widetilde{\N}_{\e} \stackrel{\beta_{\e}}{\rightarrow} M_{\delta}
	$$
	is well-defined and $\beta_{\e}\circ \Phi_{\e}$ is  homotopically equivalent to the embedding $\iota: M\rightarrow M_{\delta}$. 
	This fact and \cite[Lemma 4.3]{BC} (see also \cite[Lemma 2.2]{CLDE}) yield
	$$
	\operatorname{cat}_{\widetilde{\N}_{\e}}(\widetilde{\N}_{\e})\geq \operatorname{cat}_{M_{\delta}}(M).
	$$
	From the definition $\widetilde{\N}_{\e}$ and Proposition \ref{prop2.2}, we know that $J_{\e}$ verifies the Palais-Smale condition in $\widetilde{\N}_{\e}$ (taking $\e_{\delta}$ smaller if necessary), so we can apply  standard Ljusternik-Schnirelmann  theory for $C^{1}$ functionals (see \cite[Theorem 5.20]{W}) to obtain at least $\operatorname{cat}_{M_{\delta}}(M)$ critical points of $J_{\e}$ restricted to $\N_{\e}$. From Corollary \ref{cor2.1}, we can deduce that $J_{\e}$ has at least $\operatorname{cat}_{M_{\delta}}(M)$ critical points in $\h$.
\end{proof}

\end{document}